\documentclass[11pt,a4paper]{article}
\usepackage[T1]{fontenc}
\usepackage{amsfonts}
\usepackage{amsmath}
\usepackage{empheq}
\usepackage{amssymb}
\usepackage{wrapfig}
\usepackage{tikz}
\usepackage{pgfplots}
\pgfplotsset{compat=1.18}
\usepackage{enumitem}
\setlist{leftmargin=0mm}
\usepackage{setspace}
\usepackage{skull}
\usepackage{nicematrix}
\usepackage[english]{babel}
\usepackage{hyperref}
\usepackage{xurl}
\usepackage{comment}
\usepackage{amsthm}
\usepackage{changepage}   
\usepackage{dsfont}
\usepackage[
backend=biber,
style=numeric,
sorting=nyt,
maxbibnames=99]{biblatex}
\usepackage{graphicx}
\usepackage{caption}
\usepackage[pagewise]{lineno} 
\newtheorem{theorem}{Theorem}

\newtheorem{prop}{Proposition}
\theoremstyle{definition}

\newtheorem{lemma}{Lemma}
\newtheorem{remark}{Remark}

\usepackage{lipsum}

\newcommand\phantomarrow[2]{
  \setbox0=\hbox{$\displaystyle #1\to$}%
  \hbox to \wd0{%
    $#2\mapstochar
     \cleaders\hbox{$\mkern-1mu\relbar\mkern-3mu$}\hfill
     \mkern-7mu\rightarrow$}%
  \,}
\makeatletter
\newcommand*{\rom}[1]{\expandafter\@slowromancap\romannumeral #1@}
\makeatother

\newcommand{\R}{\mathbb{R}}

\newcommand{\N}{\mathbb{N}}

\DeclareMathOperator{\diag}{diag}

\DeclareMathOperator{\Div}{div}
\DeclareMathOperator{\cv}{\mathfrak{c}_v}

\usepackage{csquotes}
\usepackage{xcolor} 

\numberwithin{theorem}{section}
\numberwithin{equation}{section}
\numberwithin{remark}{section}
\numberwithin{prop}{section}
\numberwithin{definition}{section}
\numberwithin{lemma}{section}

\title{Pressure and temperature relaxation limit for a one-velocity Baer-Nunziato model}
\date{\today}
\author{Cosmin Burtea\thanks{cosmin.burtea@imj-prg.fr Institut de Mathématiques de Jussieu -- Paris Rive Gauche, Université de Paris, 8 Place Aurélie Nemours, 75013 Paris, France}\,, Timothée Crin-Barat\thanks{Universit\'e Paul Sabatier,  Institut de Math\'ematiques de Toulouse, Route de Narbonne 118, 31062 Toulouse Cedex 9, France, timothee.crin-barat@math.univ-toulouse.fr}\, and Pierre Gonin-{}-Joubert
\thanks{Universit\'e Claude Bernard Lyon 1, ICJ UMR5208, CNRS, \'Ecole Centrale de Lyon, INSA Lyon, Université Jean Monnet, 69622 Villeurbanne, France \texttt{goninjoubert@math.univ-lyon1.fr}}\thanks{LAMA UMR5127 CNRS, Université Savoie Mont Blanc, Le Bourget du lac, France}}
\begin{document}
\maketitle
\date{}
\begin{abstract}
The dynamics of two-phase flows out of mechanical and thermal equilibrium are described by a partially dissipative first-order quasilinear system with stiff interaction terms associated with fast relaxation scales. In this paper, we analyze from a mathematical point of view the resulting \textit{pressure and temperature relaxation} singular limit problem for a one velocity Baer-Nunziato model. This leads to a singular limit problem involving two small parameters. We propose a uniform symmetrization of this system which allows us to justify the strong relaxation limit and to establish a convergence rate for classical solutions.

\end{abstract}

\section{Introduction}
\subsection{Presentation of the model}

\ \indent The \eqref{BN} model was introduced to study deflagration-to-detonation transition in reactive granular media in \cite{Baer1986}. In all generality, it is a quasilinear system of PDEs of order one consisting of $5+2d$ equations which enjoy a rather non-standard mathematical structure: the equations are non-conservative and degenerate in the sense that the entropy naturally associated to the system is not strictly convex  \cite{Forestier2011,BurteaCrin-BaratTan2023,CordesseMassot2019}. In particular, the  Lax-Friedrichs theorem for symmetrization does not apply in this context. The \eqref{BN} system of equations is obtained via volume averaging procedure of the single-phase balance equations \cite{truesdell1984thermodynamics,DrewPassman2006,IshiiTakashi2010}, supplemented with a closure relation governing the evolution of the volume fraction in order to obtain a closed system. The resulting equations are non-conservative in form, due to interfacial exchanges between the two phases. This model has the good structure for nonhomogeneous continuous media interaction and is used to analyze multiphase flows (see for instance \cite{HerardSalehSeguin2018,Herard2005,CoquelHerardSalehSeguin2013} and the references cited within). 

\medskip

The study of the hyperbolicity of the \eqref{BN} model was addressed in
\cite{Embid1992,Saurel1999,Saurel2003,gallouet2004numerical,Forestier2011}. Furthermore, symmetrization, which implies the local well-posedness of the model at least in some regions of the phase space, was addressed in \cite{gallouet2004numerical,CoquelHerardSalehSeguin2013,SalehSeguin2020,HerardSalehSeguin2018,CordesseMassot2019}, while the Riemann problem was studied in \cite{Embid1992,Saurel1999,Saurel2003,gallouet2004numerical}. To our knowledge, the problem of symmetrizing a general non-conservative system with degenerate entropy has no answer although an interesting discussion is provided in \cite{CordesseMassot2019} for the case of non-conservative systems with strictly convex entropy.

\medskip

In this kind of model, the two phases of the mixture are out of mechanical, thermal, and kinetic equilibrium, respectively. Each of these non-equilibrium processes has a characteristic equilibration timescale: the values cited for instance in \cite{Kapila2001} are approximately 0.1 ms for kinetic equilibrium, 0.03 ms for pressure equilibrium, and 18 ms for thermal equilibrium. This leads to a system with stiff interaction terms and for computational purposes reduced models can be derived to describe the flow in regions far from the shock wave. 

\medskip

In this paper, we consider the following one-velocity, two-pressure, two-temperature model that is obtained by Kapila et al. in \cite{Kapila2001} after velocity reduction:

\begin{equation}
\left\{
\begin{array}
[c]{l}%
    \partial_t \alpha_\pm + u\,\nabla \alpha_\pm = \pm\dfrac{\alpha_+\alpha_-}{\mu\varepsilon}\,(p_+-p_-), \\ \\
    \partial_t (\alpha_\pm \rho_\pm) + \Div(\alpha_\pm\rho_\pm u) = 0, \\
    \\
    \partial_t (\rho u) + \Div (\rho u\otimes u) +\nabla (\alpha_+p_+ + \alpha_- p_-)= 0, \\ \\
    \partial_t\!\left(\alpha_\pm\rho_\pm e_\pm\right) 
      + \Div\!\left(\alpha_\pm\rho_\pm e_\pm u\right) + \alpha_\pm p_\pm \,\Div u
      =        \mp \dfrac{\alpha_+\alpha_-}{\mu\varepsilon}\,p_-\,(p_+-p_-) 
         \mp \dfrac{\theta_+-\theta_-}{\mu}.\end{array}
\right.  \tag{$BN$}\label{BN}%
\end{equation}
 The states variables of the two phases are
\begin{itemize}
\item $\alpha_\pm$, the volume fractions, such that,  $\alpha_+ + \alpha_-=1$;
\item $\rho_\pm$, the densities and we also denote by $\rho = \alpha_+\rho_+ + \alpha_-\rho_-$ the average density;
\item $u$, the velocity;
\item $p_\pm$, the pressures;
\item $e_\pm$, denote the internal energies;
\item $\theta_\pm$, denote the temperatures.
\end{itemize}
We assume that the internal energies are proportional with the temperatures $$e_{\pm}=\cv_{\pm}\theta_\pm$$ and we assume perfect gas e.o.s. $$p_\pm = R_\pm\rho_\pm\theta_\pm.$$ We have the following relations relating the perfect gas constants $R_\pm$ with the heat capacities $\cv_\pm$: $$R_\pm = \cv_\pm (\gamma_\pm - 1).$$
The adiabatic constants $\gamma_+$ and $\gamma_-$ satisfy
\begin{align}\label{cond:gamma}\gamma_+ \ne\gamma_-.\end{align}
The parameters $\mu$ and $\varepsilon$ are supposed to be small in practice. More precisely, $\mu$ and $\mu\varepsilon$ are proportional to the relaxation times associated to the return to thermal and mechanical equilibrium, respectively. As explained in \cite{Kapila2001}, the former is faster than the latter which justifies our choice of the relaxation parameters in \eqref{BN}. Throughout the analysis, the parameter $\varepsilon$ is taken to be a sufficiently small fixed constant, whereas the parameter $\mu$ is sent to $0$.

\subsection{Aims of the paper}

In this paper, we study the singular limits arising from considering \textit{both the pressure and temperature relaxation} which corresponds to $\mu\rightarrow0$. Formally, in the limit one expects to find solutions for the reduced one-pressure, one-temperature model:
\begin{equation}
\left\{
\begin{array}
[c]{l}%
   
    \partial_t (\alpha_\pm \rho_\pm) + \Div(\alpha_\pm\rho_\pm u) = 0, \\
    \\
    \partial_t (\rho u) + \Div (\rho u\otimes u) +\nabla p= 0, \\ \\
    \partial_t\!\left(\alpha_+\rho_+ e_+ +\alpha_-\rho_- e_-\right) 
      + \Div\!\left((\alpha_+\rho_+ e_+ +\alpha_-\rho_- e_-) u\right) + p\Div u
      =  0   ,\\ \\
        p=  R_+\rho_+\theta_+=R_-\rho_-\theta_- ,\ \theta_+=\theta_-, \\ \\
          e_\pm=c_{v_\pm}\theta_\pm,\
          \rho=\alpha_+\rho_++\alpha_-\rho_-.
         \end{array}
\right.  \tag{$K$}\label{Kapila}%
\end{equation}

Mathematically, this is a singular limit problem associated to a quasilinear system of equations of the form:
\begin{equation}
\partial_{t}V+\sum_{j=1}^{d}A^j\left(  V\right)  \partial_{j}V+\frac{Q\left(  V\right)}{\mu}=0,\label{Multi}\tag{$S$}%
\end{equation}
where $d\geq1$ is the space dimension, $V\in\mathbb{R}^{n}$ and for all $j\in \overline{1,d}$
$A^j:\mathbb{R}^{n}\rightarrow\mathbb{R}^{n}$.
The study of systems of the form \eqref{Multi} goes back to the pioneering work of Liu \cite{Liu1987hyperbolic}, Chen, Levermore and Liu  \cite{Chen1994} for weak solutions in $1d$, Shi and Xin for numerical aspects \cite{ShiZhouping1995} or Yong \cite{Yong1999,Yong2004} for classical solutions in any space dimension. More recently, the subject was revisited and a number of interesting general results were obtained regarding the global existence of classical solutions, we refer the reader to Kawashima and Xu \cite{XuKawashima2014b}, Beauchard and Zuazua \cite{BeauchardZuazua2011}, Crin-Barat and Danchin \cite{CBD3}, Danchin \cite{DanchinEMS}.
Specific applications include, porous media \cite{CoulombelGoudon2007,CBD3}, gases out of thermodynamical equilibrium  \cite{zeng1999gas,GiovangigliYong2015,GiovangigliYong2018} or homogeneous mixtures \cite{BoudinGrecPavan2019,GeorgiadisTzavaras2023}. 

\medskip

 In general, in order to study the behavior of solutions of \eqref{BN} in the limit $\mu\rightarrow0$, we need to put it into an appropriate \textit{normal form} : loosely speaking, this corresponds to identifying the fast and slow variables and highlight the coupled dynamics that they obey. Besides, being able to construct a symmetrization for this normal form, Yong identified \cite{Yong1999} structural assumptions that the symmetrizer and the order zero terms should obey in order to be able to justify the singular limit. He latter revisited these conditions \cite{Yong2004} in the context of conservation laws and reformulated them with respect to the conservative variables and a supplementary formally conserved strictly convex entropy that this systems should posses. 
 
\medskip

General multiphase models of Baer-Nunziato type are non-conservative and the natural associated entropy is not strictly convex \cite{Forestier2011,CordesseMassot2019}. Recall that system \eqref{BN} is obtained for instance in \cite{Kapila2001} after velocity reduction. One of the main differences with the two-velocity case is that a conservative form can be recovered, although a certain ambiguity as for which set of conservative variables are better suited still persists: for instance, putting aside for a moment the internal energy, we could work with either of the two sets of variables $(\alpha_+\rho_+,\alpha_-\rho_-, \rho \alpha)$ or $(\rho,\rho y, \rho \alpha)$. Also, the natural entropy can be augmented with a term of the type $\rho\psi(\alpha_+)$ as to recover strict convexity (see for example \cite{gallouet2004numerical}). However, verifying Yong's conditions \cite{Yong2004} and understanding the equation of the order one corrector which is needed in order to justify the singular limit is not transparent.

\medskip

Building upon our previous works \cite{BurteaCrin-BaratTan2023, BurteaCrin-BaratGonin-Joubert2026}, we put forward a symmetric normal form for \eqref{BN} which allows us to construct unique classical solutions for the system \eqref{BN} whose time of existence is bounded below uniformly with respect to both $\varepsilon$ and $\mu$. Furthermore, the normal form we obtained allows us to show that the difference between solutions of \eqref{BN} and that of \eqref{Kapila} are, loosely speaking, of the same order as the sum of the initial data and $\sqrt{\mu}$ in $H^{s-1}$ or even $\mu$ in $H^{s-2}$ where $H^s$ is the regularity of the initial data. Owing to the particular structure of the system under consideration, we are able to establish decay rates for the zeroth-order approximation, in contrast with \cite{Yong1999}, where such rates appear only at the first-order level.
\medbreak
Arguably, the most important part of our work lies in rewriting the system \eqref{BN} in a normal form that effectively decouples it into two subsystems: one governing the fast variables, and another governing the slow variables. The key observation is that the latter retains a structure that is asymptotically close to that of the limiting system as $\mu \to 0$. Loosely speaking, no singular terms are present, whereas the coupling through singular terms appears only in the subsystem associated with the fast variables. 
Due to the complexity of the \eqref{BN} system, which involves numerous equations, source terms, and unknowns, most of the algebraic derivations in the second section of this paper were performed using the SageMath symbolic computation software. The reader is invited to consult the accompanying Jupyter notebook for the full step-by-step calculations, see the \href{https://github.com/PierreGJ/Article-Appendix-Pressure-and-temperature-relaxation-limit-for-a-one-velocity-Baer-Nunziato-model}{GitHub Repository}.

\section{Strategy to reformulate \texorpdfstring{\eqref{BN}}{BN}}\label{sec:sym}
\subsection{Some heuristics}
An important quantity which obeys a simple equation is the mass fraction $y_\pm$ of the phase $\pm$ given by
\begin{equation}
y_\pm = \dfrac{\alpha_\pm\rho_\pm}{\rho}.
\end{equation}
Since $y_+ + y_- = 1$, we obtain
\begin{equation}\label{eq:cand1}
\partial_t y_\pm + u\cdot\nabla y_\pm = 0.
\end{equation}
In the following, we will use the $\langle \cdot \rangle$-notation for mass-averaged quantities: for functions $f_\pm$, we will denote
\begin{equation}
\langle f \rangle = y_+ f_+ + y_- f_-.
\end{equation}
As a consequence,
\begin{equation}
\partial_t (\rho \langle f \rangle) + \Div(\rho \langle f \rangle u) = y_+ \partial_t(\rho f_+) + y_+ \Div(\rho f_+ u) + y_- \partial_t(\rho f_-) + y_- \Div(\rho f_- u).
\end{equation}
From \eqref{BN}, we see that
\begin{equation}\label{eq:souris}
\partial_t(\rho \langle e \rangle) + \Div(\rho \langle e \rangle u) = - (\alpha_+p_+ + \alpha_- p_-)\Div u.
\end{equation}
Thus, denoting
\begin{equation}
E = \langle e \rangle + \frac{u^2}{2},
\end{equation}
we get the global conservation of total energy:
\begin{equation}\label{eq:energ}
\partial_t(\rho E) + \Div(\rho E u) = 0.
\end{equation}

We now define the entropies for each phase:
\begin{equation}
s_\pm = \cv_\pm \ln\theta_\pm - R_\pm \ln \rho_\pm.
\end{equation}
We have
\begin{empheq}[left=\empheqlbrace]{align}
\alpha_+\rho_+\theta_+ D_t s_+ &= \dfrac{\alpha_+\alpha_-}{\mu\varepsilon}(p_+ - p_-)^2 - \frac{\theta_+-\theta_-}{\mu}\label{eq:chat1},\\
\alpha_-\rho_-\theta_- D_t s_- &= \frac{\theta_+ - \theta_-}{\mu}
\label{eq:chat2}.
\end{empheq}
Dividing \eqref{eq:chat1} by $\theta_+$, \eqref{eq:chat2} by $\theta_-$ and summing the resulting equalities, we get
\begin{equation}\label{eq:crayon}
\partial_t(\rho \langle s\rangle) +  \Div (\rho \langle s\rangle u)  = \alpha_+\alpha_-\frac{(\delta p)^2}{\mu\varepsilon} + \frac{1}{\theta_+\theta_-}\frac{(\delta\theta)^2}{\mu},
\end{equation}
where $\delta p=p_+-p_-$ and $\delta \theta=\theta_+-\theta_-$.
Introducing the function $h(x)= x - 1 - \ln(x)\geq 0$, we remark that
\begin{equation}\label{eq:magiceg}
\rho \langle \cv h(\theta)\rangle + \rho \langle  R h(1/\rho)\rangle + \rho \langle s \rangle + \alpha_- R_+ + \alpha_+ R_- = \rho \langle e\rangle - \rho \langle \cv \rangle - \rho \langle R \rangle + R_+ +  R_-.
\end{equation}
Hence, denoting 
\begin{equation}
    \eta = \langle \cv h(\theta) \rangle + \langle R h(1/\rho)\rangle + \frac{u^2}{2} + \langle \gamma\cv\rangle + \frac{\alpha_-}{\rho R_+} + \frac{\alpha_+}{\rho R_-},
\end{equation}
we observe that $\eta\geq 0$ and
\begin{equation}\label{eq:dissipeta}
   \partial_t(\rho\eta) + \Div(\rho\eta u) + \alpha_+\alpha_-\frac{(\delta p)^2}{\mu\varepsilon} + \frac{1}{\theta_+\theta_-}\frac{(\delta\theta)^2}{\mu} = 0.
\end{equation}
Equation \eqref{eq:dissipeta} motivates the fact that the unknowns $\delta p$ and $\delta\theta$ are naturally dissipated by the system.


\subsection{Strategy}\label{subsect:strat}
In order to show that the \eqref{BN} system is locally well-posed, we will prove that it is Friedrich-symmetrizable.
Let us explain in the following lines how one can exploit the one-velocity transport structure to symmetrize \eqref{BN}. In this one-velocity setting, there exists $A = A(\alpha_+, \rho_\pm, \theta_\pm)$ such that the \eqref{BN} system reads
\begin{equation}\label{bn2}
\left\{
\begin{array}
[c]{l}%
\partial_t y_+ + u\cdot \nabla y_+ = 0,
\\
\partial_t U(\alpha_+,\rho_\pm,\theta_\pm) + u\cdot\nabla U(\alpha_+,\rho_\pm,\theta_\pm) + (\Div u)A(\alpha_+,\rho_\pm,\theta_\pm) = O\left(\dfrac{\delta p}{\mu\varepsilon}, \dfrac{\delta\theta}{\mu}\right),
\\
\partial_t u + u\cdot\nabla u + \dfrac{\nabla p}{\rho} = 0,\end{array}\right.
\end{equation}
where $U(\alpha_+,\rho_\pm,\theta_\pm)$ corresponds to a set of unknowns to determine, and for some smooth functions $f$ and $g$, $O\left(f, g\right)$ represent terms of the form \begin{equation}\label{grandO}
b(\alpha_+,\rho_\pm,\theta_\pm)f + c(\alpha_+,\rho_\pm,\theta_\pm)g
\end{equation}
for some functions $b$ and $c$.
To symmetrize \eqref{bn2}, we only need to focus on the terms $(\textrm{div} u)A$ and $\nabla p/\rho$. Since the unknown $u$ is fixed, we begin by decomposing the pressure term in order to clarify how to select the remaining unknowns.
Observing that
\begin{equation}
\dfrac{\nabla p}{\rho}=\nabla\left(\frac{p}{\rho}\right) + \frac{p}{\rho^2}\nabla\rho,
\end{equation}
we are led to consider $\rho$ as an unknown in the new system. As we seek to introduce three additional unknowns in order to close the system, we proceed by decomposing
\begin{equation}\label{eq:cond1}
\frac{p}{\rho} = U_1 + U_2 + w.
\end{equation}
Suppose that we have constructed a change of variables
\begin{equation}
(\alpha_+,\rho_\pm,\theta_\pm) \rightarrow (U_1,U_2,w,y_+,\rho).
\end{equation}
Then, there exist coefficients $a_i = a_i(\alpha_\pm,\rho_\pm,\theta_\pm)$ such that (with the notation $U_3=w$) :
\begin{equation}\label{eq:port}
\text{for all }1 \leq i \leq 3,\quad \partial_t U_i + u\cdot\nabla U_i + a_i\Div u = O\left(\frac{\delta p}{\mu\varepsilon},\frac{\delta\theta}{\mu}\right).
\end{equation}
In the one-dimensional setting, where \(d = 1\), the matrix formulation would take the form:
\begin{equation}\label{eq:superform}
\partial_t\begin{bmatrix}
U_1\\
U_2\\
w\\
y_+\\
\rho\\
u
\end{bmatrix}
+
\begin{bmatrix}
u & 0 & 0 & 0 & 0 & a_1\\
0 & u & 0 & 0 & 0 & a_2\\
0 & 0 & u & 0 & 0 & a_3\\
0 & 0 & 0 & u & 0 & 0\\
0 & 0 & 0 & 0 & u & \rho\\
1 & 1 & 1 & 0 &p/\rho^2& u \\
\end{bmatrix}
\partial_{x}\begin{bmatrix}
U_1\\
U_2\\
w\\
y_+\\
\rho\\
u 
\end{bmatrix}
= O\left(\frac{\delta p}{\mu\varepsilon},\frac{\delta\theta}{\mu}\right).
\end{equation}
Thus, such a system is symmetrizable if
\begin{equation}\label{eq:cond3}
\text{for all } 1\leq i \leq 3,\quad a_i>0.
\end{equation}
A positive definite symmetric matrix $S$ to symmetrize the system will simply be
\begin{equation}
S = \diag(1/a_1, 1/a_2, 1/a_3, 1,p/\rho^2, \rho).
\end{equation}
We emphasize here that the same heuristics hold for $d\geq 1$.

In the next section, we explain how to construct the variables \(U_1\), \(U_2\) and \(w\), choosing them so as to symmetrize the system while carefully managing the structure of the right-hand side of \eqref{eq:superform}, in order to recover estimates that are uniform with respect to the parameters \(\varepsilon\) and \(\mu\).
Recall that we have already chosen to choose to work with the variables $\rho, y_+$ and $u$ because their evolution equation do not contain any singular terms.
\subsection{Symmetrization adapted to the relaxation mechanisms}\label{subsec:sym}
Rewriting \eqref{eq:souris} as
\begin{equation}
\partial_t \langle e\rangle + u\nabla\langle e\rangle + \frac{p}{\rho}\Div u = 0,
\end{equation}
we remark that the unknown $\langle e \rangle$ is a natural candidate. Indeed, its equation does not contain any singular terms and $p/\rho$ has the same physical dimension as $\langle e\rangle$. Let us now explain how to relate $e$ to $p/\rho$.
Recall that for a perfect gas, the pressure satisfies
\[
p = (\gamma - 1)\rho e.
\]
Starting from the relation
\begin{equation}
p = \alpha_+ \rho_+ R_+ \theta_+ + \alpha_- \rho_- R_- \theta_-,
\end{equation}
and defining the averaged temperature
\[
\theta = \alpha_+ \theta_+ + \alpha_- \theta_-,
\]
we obtain
\begin{equation}\label{eq:chips}
p = \langle R \rangle \rho \theta + O(\delta \theta).
\end{equation}
Besides, we have
\begin{align}
\rho\langle e \rangle &= \alpha_+\rho_+\cv_+\theta_+ + \alpha_-\rho_-\cv_-\theta_-
\\&=\langle \cv \rangle \rho\theta + O(\delta\theta)\label{eq:choups}.
\end{align}
Thus \eqref{eq:chips} and \eqref{eq:choups} give
\begin{equation}
\frac{p}{\rho} = \frac{\langle R\rangle}{\langle \cv\rangle} \langle e\rangle + O(\delta\theta).
\end{equation}
Denoting $w := \dfrac{\langle R\rangle}{\langle\cv\rangle}\langle e\rangle$, we have
\begin{equation}
    \partial_t w + u\cdot\nabla w + \frac{\langle R\rangle}{\langle \cv\rangle}\frac{p}{\rho}\Div u = 0.
\end{equation}
Concerning the choice of $U_1$ and $U_2$, recalling \eqref{eq:cond1}, we get
\begin{equation}
   U_1+U_2= \frac{p}{\rho} - w = (\gamma_+ - \gamma_-)\frac{\cv_+\cv_-}{\langle \cv\rangle}y_+ y_- \delta\theta.
\end{equation}
Consequently, $U_1$ and $U_2$ will be functions that vanish as $\delta \theta$ approaches zero, i.e., as $\mu\to0$. Let us now examine the equation satisfied by \(p/\rho - w\). From \eqref{BN}, it follows that
\begin{align}
    D_t \left(\frac{p}{\rho}-w\right)& + \frac{\cv_+\cv_-}{\langle\cv\rangle}(\gamma_+-\gamma_-)y_+y_-((\gamma_+ - 1)\theta_+ - (\gamma_--1)
    \theta_-)\Div u + \frac{1}{\mu}\frac{\langle\cv\rangle}{\cv_+\cv_-}\frac{\rho}{y_+ y_-}\left(\frac{p}{\rho} - w\right)\nonumber\\&= - \frac{(\gamma_+-\gamma_-)\alpha_+\alpha_-p_-\delta p}{\rho\varepsilon\mu}.
\end{align}
One observes that $p/\rho - w$ is damped at rate $1/\mu$; however, no dissipative mechanism is available to control the source term of order $1/\varepsilon\mu$. This suggests that one should identify some dissipative effect acting on \(\delta p\) at the same order. From \eqref{BN}, we have
\begin{equation}
    D_t (\delta p)  + (\gamma_+ p_+ - \gamma_- p_-)\Div u +\frac{((\gamma_+ \alpha_- + {\gamma_-} \alpha_+) p_- + \alpha_-\delta p)\delta p}{\varepsilon \mu} 
= -\left(\frac{\gamma_+-1}{\alpha_+} + \frac{\gamma_--1}{\alpha_-}\right)\frac{\delta\theta}{\mu}.
\end{equation}
Thus, \(\delta p\) is dissipated at rate $1/\mu\varepsilon$, provided that the source term of order $1/\mu$ is under control. The idea for finding the two missing variables is as follows: we seek a function \(f = f(\alpha_+,\rho_\pm,\theta_\pm)\) such that
\begin{equation}
U_1 := f\,\delta p,
\end{equation}
and such that, defining
\begin{equation}
U_2 := \frac{p}{\rho} - w - U_1,
\end{equation}
the equation satisfied by \(U_2\) no longer contains singular terms of order \((\mu\varepsilon)^{-1}\), canceling the contribution involving \(\delta p/(\mu\varepsilon)\). There are infinitely many possible choices for \(f\). Indeed, if \(f\) is admissible, then \(f+\delta p\) is admissible as well. Choosing
\begin{equation}
    f = \frac{\gamma_+ - \gamma_-}{\rho\left(\dfrac{\gamma_+}{\alpha_+}+\dfrac{\gamma_-}{\alpha_-}\right)},
\end{equation}
we obtain
\begin{align}
    D_t U_1 &+ \frac{(\gamma_+-\gamma_-)(\gamma_+p_+ - \gamma_-p_- - \delta p)\alpha_+\alpha_-}{\rho(\alpha_+\gamma_- + \alpha_-\gamma_+)}\Div u +(\gamma_+\alpha_- + \gamma_-\alpha_+)\left(p_- + \frac{\gamma_-\alpha_+\delta p}{(\gamma_+\alpha_- + \gamma_-\alpha_+)^2}\right) \frac{U_1}{\varepsilon\mu}\nonumber
    \\&= -\frac{\langle \cv\rangle}{\cv_+\cv_-y_+y_-}\left(\frac{(\gamma_+-1)\alpha_- + (\gamma_--1)\alpha_+}{\rho(\alpha_+\gamma_- + \alpha_-\gamma_+)}\right)\left(\frac{U_1}{\mu} + \frac{U_2}{\mu}\right), \nonumber
\end{align}
and for $U_2$,
\begin{align*}
    &D_t U_2 +\frac{\langle\cv\rangle}{\rho\cv_+\cv_-y_+y_-(\gamma_+\alpha_- + \gamma_-\alpha_+)}\frac{U_2}{\mu} 
    \\&+ \frac{(\gamma_+-\gamma_-)\cv_+\cv_-y_+y_-}{(\gamma_+\alpha_+ + \gamma_-\alpha_-)\langle\cv\rangle}\left((\alpha_+p_+ + \alpha_-p_-)(\gamma_-(\gamma_+-1)\frac{\theta_+}{p_+} - \gamma_+(\gamma_--1)\frac{\theta_-}{p_-})+\frac{\alpha_+\alpha_-\langle\cv\rangle\delta p}{\cv_+\cv_- y_+y_-\rho}\right)\Div u 
    \\&   =\frac{\rho\gamma_-}{(\gamma_+-\gamma_-)\alpha_-}\frac{U_1^2}{\mu\varepsilon} - \frac{\langle\cv\rangle}{\rho\cv_+\cv_-y_+y_-(\gamma_+\alpha_- + \gamma_-\alpha_+)}\frac{U_1}{\mu}.
\end{align*}
Recalling the notation \eqref{grandO}, remarking that $p_\pm = \rho w+ O(U_1,U_2)$, $\theta_\pm = w/\langle R\rangle + O(U_1,U_2)$, and $\alpha_\pm = y_\pm R_\pm/\langle R\rangle + O(U_1,U_2)$,  one obtains
\begin{align*}
    D_t U_1 &+ \left(\frac{(\gamma_+-\gamma_-)^2}{\langle R\rangle \langle R/\gamma\rangle}\frac{y_+R_+}{\gamma_+}\frac{y_-R_-}{\gamma_-}w+O(U_1,U_2)\right)\Div u + \left(\frac{\langle R/\gamma\rangle}{\langle R\rangle}\gamma_+\gamma_-\rho w+O(U_1,U_2)\right)\frac{U_1}{\varepsilon\mu}
    \\&=- \left(\frac{\langle \cv\rangle^2}{\rho\langle R/\gamma\rangle}\frac{(\gamma_+-1)(\gamma_--1)}{\gamma_+\gamma_-\cv_+\cv_-y_+y_-}+O(U_1,U_2)\right)\left(\frac{U_1}{\mu}+\frac{U_2}{\mu}\right),
\end{align*}
\begin{align*}
    D_t U_2 &+ \left(\frac{(\gamma_+-\gamma_-)^2}{\langle\cv\rangle \langle \gamma R\rangle}\cv_+\cv_-y_+y_- w + O(U_1,U_2)\right)\Div u +\left( \frac{\langle \cv\rangle\langle R\rangle}{\langle R/\gamma\rangle}\frac{1}{\rho\cv_+\cv_-y_+y_-\gamma_+\gamma_-}+O(U_1,U_2)\right)\frac{U_2}{\mu}
    \\&= -\left( \frac{\langle \cv\rangle\langle R\rangle}{\langle R/\gamma\rangle}\frac{1}{\rho\cv_+\cv_-y_+y_-\gamma_+\gamma_-}+O(U_1,U_2)\right)\frac{U_1}{\mu} + \left(\frac{\langle R\rangle}{\gamma_+-\gamma_-}\frac{\gamma_-}{R_-y_-}+O(U_1,U_2)\right)\frac{U_1^2}{\mu\varepsilon}
\end{align*}
and
\begin{align}
    D_t w + \left(\frac{\langle R\rangle}{\langle\cv\rangle}w +O(U_1,U_2)\right)\Div u = 0.
\end{align}

The previous analysis allows us to rewrite the system \eqref{BN} in a more compact form as
\begin{equation}\label{eq:canon0}
\partial_t U + \sum_{j=1}^d H^j(U)\partial_{x_j} U + D_{\varepsilon,\mu}U = R_{\varepsilon,\mu}
\end{equation}
where
\begin{equation}
\begin{cases} U = \begin{pmatrix}
U_1,U_2,w,y_+,\rho,u
\end{pmatrix}^T,\\
H^j = u_j I_{5+d} + L^j + C^j,
\end{cases}
\end{equation}
with the matrices $L^j,C^j\in {\cal M}_{5+d,5+d}(\R)$ given by

\begin{equation}
   L^j_{i,k} = \begin{cases}
    1 & \text{if }i = 5+j \text{ and }k\in\{1,2,3\},\\
    p/\rho^2 & \text{if }i=5+j \text{ and }k=4+j,\\
    0 & \text{else},
    \end{cases}
\end{equation}
and
\begin{equation}
    C^j_{i,k} = \begin{cases}
    \dfrac{(\gamma_+-\gamma_-)^2}{\langle R\rangle \langle R/\gamma\rangle}\dfrac{y_+R_+}{\gamma_+}\dfrac{y_-R_-}{\gamma_-}w+O(U_1,U_2) & \text{if }k=5+d \text{ and }i=1,\\
    \dfrac{(\gamma_+-\gamma_-)^2}{\langle\cv\rangle \langle \gamma R\rangle}\cv_+\cv_-y_+y_- w + O(U_1,U_2) & \text{if }k=5+d \text{ and }i=2,\\
    \dfrac{\langle R\rangle}{\langle\cv\rangle}w +O(U_1,U_2) & \text{if }k=5+d \text{ and }i=3,\\
    0 & \text{else}.
    \end{cases}
    \end{equation}
    The order zero terms in \eqref{eq:canon0} are given by
    \begin{equation}
     D_{\varepsilon,\mu} = \diag\left(\dfrac{\dfrac{\langle R/\gamma\rangle}{\langle R\rangle}\gamma_+\gamma_-\rho w}{\varepsilon\mu},\dfrac{\dfrac{\langle \cv\rangle\langle R\rangle}{\langle R/\gamma\rangle}\dfrac{1}{\rho\cv_+\cv_-y_+y_-\gamma_+\gamma_-}}{\mu},0,\cdots,0\right)
    \end{equation}
    and
\begin{equation}
    R_{\varepsilon,\mu} = \begin{pmatrix}
        r_{11}\dfrac{U_1}{\varepsilon\mu} + r_{12} \dfrac{U_2}{\mu} + r_{13}\dfrac{U_1}{\mu},
        r_{21}\dfrac{U_1^2}{\varepsilon\mu} + r_{22}\dfrac{U_2}{\mu} + r_{33}\dfrac{U_1}{\mu},0, \cdots 0
    \end{pmatrix}^T,
\end{equation}
where $r_{ii}(U) = r_{ii}^1(U) U_1 + r_{ii}^2(U) U_2$ for $i\in \{1,2\}$.

Thanks to this new formulation, we will show that if \(U_1\) and \(U_2\) are sufficiently small, then \eqref{eq:canon0} is symmetrizable by means of a diagonal matrix; see Section~\ref{subsect:strat}.
\begin{remark}
The third term $D_{\varepsilon,\mu}U$ in the left hand side of \eqref{eq:canon0} is a damping term, which will allow us to control the rest $R_{\varepsilon,\mu}$. Note that the damping effects are preserved under symmetrization, since the symmetrizing matrix is diagonal.
\end{remark}

\begin{remark} 
 We emphasize that the specific choice of the relaxation parameters, \(\mu\varepsilon\)  for the pressure relaxation and \(\mu\) for the temperature relaxation, is crucial for controlling the remainder term linear with respect to $U_2/\mu$ in the equation of $U_1$; in particular, it is necessary that the pressure relaxation occurs at a higher order, which is ensured by the fact that $\varepsilon$ is chosen small enough. This is consistent with the discussion in \cite{Kapila2001}. 
\end{remark}

Of course, the reduced limit system \eqref{Kapila} can also be written with respect to the variables $\left(  w,\rho,y\right)  $ where
\[
w=\frac{\left\langle R\right\rangle }{\left\langle \cv\right\rangle
}\left\langle e\right\rangle .
\]
One has
\begin{equation}
\left\{
\begin{array}
[c]{r}%
\partial_{t}w+u\cdot\nabla w+\tfrac{\left\langle R\right\rangle }{\left\langle
\cv\right\rangle }w\operatorname{div}u=0,\\
\partial_{t}y+u\cdot\nabla y=0,\\
\partial_{t}\rho+\operatorname{div}\left(  \rho u\right)  =0,\\
\partial_{t}u+u\cdot\nabla u+\nabla w+\frac{w}{\rho}\nabla\rho=0.
\end{array}
\right.  \label{Kapp_sym}%
\end{equation}
We recall that $\left\langle \cdot\right\rangle $ represents mass-averaged quantities: we have
\[
\left\langle R\right\rangle =y_{+}R_{+}+y_{-}R_{-} \quad \text{and} \quad \left\langle
\cv\right\rangle =y_{+}\cv_{+}+y_{-}\cv_{-}.
\]
In the next section, we show that the change of variables $$\left(  \alpha_{+},\rho_{+},\theta_{+}\right)  \rightarrow\left(
y,\rho,w\right)  $$ defines a $C^\infty$-diffeomorphism of $(0,1)\times(0,+\infty)$ onto itself.

\subsection{Global inversion}
In this section, we show that the change of variables introduced in the previous section is invertible, at least for sufficiently small pressure and temperature differences. More precisely, the aim of this subsection is to establish the following proposition.
\begin{prop}
\label{change_of_var11}
Let ${\cal U}\subset\subset (0,1)\times (0,+\infty)^4$. There exists some $\eta=\eta({\cal U})>0$ such that
\begin{equation}
\Phi :\quad
\begin{array}{ccc}
    {\cal U^\eta} &\rightarrow & \Phi({\cal U}^\eta)  \\
    (\alpha_+,\rho_\pm,\theta_\pm) &\mapsto  & (U_1,U_2,w,\rho,y), 
\end{array}
\end{equation}
is a $C^\infty$-diffeomorphism, where 
\begin{equation}
    {\cal U}^\eta = {\cal U}\cap \{\vert R_+\rho_+\theta_+ - R_-\rho_-\theta_-\vert<\eta, \vert \theta_+-\theta_-\vert <\eta\}.  
\end{equation}
\end{prop}
Roughly speaking, \(\Phi\) is invertible on a certain open set, provided that \(U_1\) and \(U_2\) are sufficiently small. $\Phi$ is $C^\infty$ on $(0,1)\times (0,+\infty)^4\times\R^d$, and for all $(\alpha_+,\rho_\pm,\theta_\pm,u)\in (0,1)\times(0,+\infty)^4\times\R$,
\begin{equation}\label{deteq}
    \det D\Phi(\alpha_+,\rho_\pm,\theta_\pm,u) = (\gamma_+-\gamma_-)^2\alpha_+\alpha_-\frac{\gamma_-\alpha_+p_+ + \gamma_+\alpha_-p_- }{(\gamma_-\alpha_+ + \gamma_+\alpha_-)^2}\frac{\cv_+\cv_-y_+y_-\langle R\rangle}{\rho\langle \cv\rangle} \neq 0.
\end{equation}
Detailed symbolic verification of \eqref{deteq} is provided in the supplementary Jupyter notebook (\href{https://github.com/PierreGJ/Article-Appendix-Pressure-and-temperature-relaxation-limit-for-a-one-velocity-Baer-Nunziato-model}{GitHub Repository}).
Thus, \(\Phi\) is locally invertible on \((0,1) \times (0,+\infty)^4 \times \mathbb{R}\), although this does not imply global invertibility. Nevertheless, we can show that the restriction of \(\Phi\) obtained by imposing \(\delta p = \delta \theta = 0\) and \(U_1 = U_2 = 0\) is invertible.

To begin with, note that the function
\begin{equation}
\begin{array}{ccc}
    (0,+\infty)^4 &\rightarrow & \{(p_+,\theta_+,\delta p,\delta\theta)\in (0,+\infty)^2\times\R^2, \vert \delta p\vert<p, \vert \delta\theta\vert < \theta\}\\
     (\rho_+,\rho_-,\theta_+,\theta_-)& \mapsto &(R_+\rho_+\theta_+,\theta_+,R_+\rho_+\theta_+-R_-\rho_-\theta_-,\theta_+-\theta_-)
\end{array}
\end{equation}
is a $C^\infty$-diffeomorphism. Indeed, this function is $C^\infty$ and invertible, with inverse
\begin{equation}
    (p_+,\theta_+,\delta p,\delta\theta)\mapsto \left(\frac{p_+}{R_+\theta_+},\frac{p_+-\delta p}{R_-(\theta_+-\delta\theta)},\theta_+, \theta_+-\delta\theta\right).
\end{equation}
Suppose that  $p_+ = p_-$ and $\theta_+=\theta_-$. Then 
\begin{equation}
    y_+ = \dfrac{\alpha_+/R_+}{\alpha_+/R_+ + \alpha_-/R_-},\quad \rho = \left(\frac{\alpha_+}{R_+}+\frac{\alpha_-}{R_-}\right)\frac{p_+}{\theta_+}\quad \text{and}\quad w = \frac{1}{\alpha_+/R_+ + \alpha_-/R_-}\theta_+.
\end{equation}
Let us define $\Psi : (0,1)\times (0,+\infty)^2\rightarrow (0,1)\times(0,+\infty)^2$ by
\begin{equation}
    \Psi(\alpha,p,\theta) = \left(\frac{\alpha/R_+}{\alpha/R_+ + (1-\alpha)/R_-},\left(\frac{\alpha}{R_+}+\frac{1-\alpha}{R_-}\right)\frac{p}{\theta},\frac{\theta}{\alpha/R_+ + (1-\alpha)/R_-}\right).
\end{equation}
Then $\Psi$ is a $C^\infty$-diffeomorphism with inverse
\begin{equation}
    \Psi^{-1}(y,\rho,w) = \left(\frac{yR_+}{yR_+ + (1-y)R_-},\rho w,\dfrac{w}{yR_+ + (1-y)R_-}\right),
\end{equation}
We may then apply the following lemma to conclude:
\begin{lemma}\label{prop:lemmainv}
    For $Z\in\R^n$, we denote $Z = (Z_1,Z_2)\in \R^p\times\R^{n-p}$. Let ${\cal Z} = {\cal Z}_1\times {\cal Z}_2\subset\R^p\times \R^{n-p}$ be a bounded open set, such that $0\in {\cal Z}_2$. Let $\Phi \in C^1({\cal Z},\R^n)$ and $\pi\in C^1(\R^n,\R^p)$. Suppose that, for all $Z_1\in {\cal Z}_1$, $\det D\Phi((Z_1,0))\neq 0$. Assume moreover that $\Psi : {\cal Z}_1\rightarrow \Psi({\cal Z}_1)$, $Z_1\mapsto \Psi(Z_1)=\pi(\Phi(Z_1, 0))$ is a $C^1$-diffeomorphism. Then, for all ${\cal V}={\cal V}_1\times {\cal V}_2\subset\subset {\cal Z}$, there exists $\eta>0$ such that $\Phi : {\cal V}^\eta \rightarrow \Phi({\cal V}^\eta)$ is a $C^1$-diffeomorphism, where ${\cal V}^\eta = {\cal V}\cap \{\vert Z_2\vert < \eta\}$.
\end{lemma}
\begin{proof}
As $\Phi\in C^1({\cal Z},\R^n)$, $\Phi$ is locally Lipschitz continuous on $\overline{{\cal V}}$, so globally Lipschitz continuous due to the compactness of $\overline{{\cal V}}$. Denote by $L>0$ the associated Lipschitz constant and let $F=\overline{\cal V}\cap\{Z_2 = 0\}$. For all $Z\in F$, $\det D\Phi(Z)\neq 0$, so there exists an open set ${\cal W}^Z\subset {\cal Z}$ such that $Z\in {\cal W}^Z$ and $\Phi : {\cal W}^Z\rightarrow \Phi({\cal W}^Z)$ is a $C^1$-diffeomorphism. Denote ${\cal W}:= \cup_{Z\in F} {\cal W}^{Z}$ and $\eta_0=d({\cal W}^c,F)/3>0$. Then ${\cal V}^{{\eta_0}}\subset\subset {\cal W}$. Indeed, if $Z\in {\cal V}^{\eta_0}$, then $\vert Z_2\vert\leq \eta_0$ so $(Z_1, g(Z_1))\in F + \overline{B(0,\eta_0)}$. Thus $d(Z, F+\overline{B(0,\eta_0)})\leq \eta_0$ and $d(Z,F)\leq 2\eta_0 < d({\cal W}^c,F)$ follows. By the Lebesgue's number lemma applied to the compact $\overline{{\cal V}^{\eta_0}}$, there exists $r>0$ such that, if $Z, Z'\in {\cal V}^{\eta_0}$ and $\vert Z - Z'\vert\leq r$, then there exists $V\in F$ such that $Z,Z'\in {\cal W}^{V}$. Let $\eta\in ]0,\eta_0]$ to be chosen and assume that $\Phi(Z) = \Phi(Z')$ for some $Z, Z'\in {\cal V}^\eta$. We have
\begin{align}
    \vert\Psi(Z_1) - \Psi(Z_1')\vert &= \vert \pi(\Phi(Z_1,0)) - \pi(\Phi(Z_1',0))\vert
    \\&\leq \sup_{V\in \Phi(\overline{\cal V})} \Vert D \pi(V)\Vert\vert \Phi(Z_1,0) - \Phi(Z_1', 0)\vert.
\end{align}
Moreover,
\begin{align}
    \vert \Phi(Z_1,0) - \Phi(Z_1',0) \vert &= \vert \Phi(Z_1,0) - \Phi(Z_1,Z_2) + \Phi(Z_1',Z_2') - \Phi(Z_1',0)\vert 
    \\&\leq L(\vert Z_2\vert + \vert Z_2' \vert) 
    \\&\leq 2\eta L.
\end{align}
Finally,
\begin{equation}
    \vert Z_1 - Z_1'\vert \leq 2\eta L\sup_{V\in \Psi(\overline{\cal V})} \Vert D\Psi^{-1}(V)\Vert \sup_{V\in \Phi(\overline{\cal V})} \Vert D \pi(V)\Vert,
\end{equation}
hence
\begin{align}
    \vert Z - Z'\vert &\leq \vert Z_1 - Z_1'\vert + \vert Z_2\vert + \vert Z_2'\vert  
    \\&\leq 2\eta + 2\eta L\sup_{V\in \Psi(\overline{\cal V})} \Vert D\Psi^{-1}(V)\Vert \sup_{V\in {\Phi(\overline{\cal V})}} \Vert D \pi(V)\Vert
    \\&\leq \eta C.
\end{align}
Choosing $\eta\leq \min(r/{C},\eta_0)$, we have $\vert Z-Z'\vert\leq r$, so there exists $V\in F$ such that $Z,Z'\in {\cal W}^V$. Since $\Phi$ is injective on ${\cal W}^V$, it follows that $Z = Z'$. Finally, $\Phi_{\vert{\cal V}^\eta}$ is injective; thus the proof of Lemma \ref{prop:lemmainv} is complete. 
\end{proof}

\section{Main results}
\ \ \ \ As announced in the introduction, we study the behavior of classical solutions to the system~\eqref{BN} in the limit $\mu\to0$. Some preliminary discussions are needed before stating our main result. The unknowns will be sought as perturbations (not necessary small) around constant equilibrium states. Below, we
use the classical notation for the $L^{2}-$based Sobolev space $H^{s}(
\mathbb{R}^{d})  $ with $d\geq1$ and
\begin{equation}
s>\frac{d}{2}+1.\label{cond_s}%
\end{equation}
Consider the constant equilibrium state
\begin{equation}
\left(  \bar{\alpha}_{+},\bar{p},\bar{\theta}\right)
\in(0,1)\times(0,+\infty)^{2}.\label{volume_fraction_pressure_temp}%
\end{equation}
In the following, we reserve the overline notation for the asymptotic states as \(\lvert x \rvert \to \infty\). We define
\begin{equation}
\bar{\rho}_{\pm}=\frac{\bar{p}}{R_{\pm}\bar{\theta}},\text{
}\bar{\rho}=\bar{\alpha}_{+}\text{
}\bar{\rho}_{+}+\bar{\alpha}_{-}\text{
}\bar{\rho}_{-},\text{ }\bar{y}_{\pm}=\frac{\bar{\alpha}_{\pm%
}\text{
}\bar{\rho}_{\pm}}{\bar{\rho}},\text{ }\bar{\alpha}_{-}%
=1-\bar{\alpha}_{+}.\label{other_constants}%
\end{equation}
Recall that $\bar{\rho}$ represents the mixture density while $\bar{y}_\pm$ are
the mass fractions and that we are given positive constants $R_\pm,c_{v_\pm}>0.$

Concerning the initial data, we fix a compact set
\begin{equation}
K\subset\left(  0,1\right)  \times\left(  0,\infty\right)  ^{2},
\label{compact_set_K}%
\end{equation}
a positive constant $M>0$ and for any $\varepsilon>0,\mu>0,$ we consider%
\begin{equation}
\left(  \alpha_{+,0}^{\varepsilon,\mu},p_{0}^{\varepsilon,\mu},\theta_{0}%
^{\varepsilon,\mu},u_{0}^{\varepsilon,\mu}\right)  \in\left(  \bar{\alpha}_{+%
},\bar{p},\bar{\theta},0_{\mathbb{R}^{d}}\right)  +(H^{s}%
(\mathbb{R}^{d}))^{3+d}\label{condition_initiala}%
\end{equation}
such that
\begin{equation}
\left(  \alpha_{+,0}^{\varepsilon,\mu},p_{0}^{\varepsilon,\mu},\theta
_{0}^{\varepsilon,\mu}\right)  \in K,\text{ }\Vert(\alpha_{+,0}^{\varepsilon,\mu}-\bar{\alpha}_{+}%
,p_{0}^{\varepsilon,\mu}-\bar{p},\theta_{0}^{\varepsilon,\mu}-\bar{\theta})\Vert_{H^{s}}\leq M.
\label{uniform_from_above_below}%
\end{equation}
We also consider
\begin{equation}
(\delta p_{0}^{\varepsilon,\mu},\delta\theta_{0}^{\varepsilon,\mu})\in
(H^{s}(  \mathbb{R}^{d})  )^{2}\label{perturbations}%
\end{equation}
and construct
\begin{equation}
\left\{
\begin{array}
[c]{r}%
\theta_{+,0}^{\varepsilon,\mu}=\theta_{0}^{\varepsilon,\mu}+\alpha
_{-,0}^{\varepsilon,\mu}\delta\theta_{0}^{\varepsilon,\mu},\\
\theta_{-,0}^{\varepsilon,\mu}=\theta_{0}^{\varepsilon,\mu}-\alpha
_{+,0}^{\varepsilon,\mu}\delta\theta_{0}^{\varepsilon,\mu},\\
R_{+}\rho_{+,0}^{\varepsilon,\mu}\theta_{+,0}^{\varepsilon,\mu}=p_{0}%
^{\varepsilon,\mu}+\alpha_{-,0}^{\varepsilon,\mu}\delta p_{0}^{\varepsilon
,\mu},\\
R_{-}\rho_{-,0}^{\varepsilon,\mu}\theta_{-,0}^{\varepsilon,\mu}=p_{0}%
^{\varepsilon,\mu}-\alpha_{+,0}^{\varepsilon,\mu}\delta p_{0}^{\varepsilon
,\mu}.
\end{array}
\right.  \text{ }\label{real_initial_data}%
\end{equation}
As soon as $\delta\theta_{0}^{\varepsilon,\mu}$ is small enough, the
temperatures will be strictly positive and $\rho_{\pm,0}^{\varepsilon,\mu}$ will
be well-defined. 

\medbreak

We are now in the position of stating our first main result which concerns the
existence of a solution on a time scale independent of the singular parameters.

\begin{theorem}
[Uniform local well-posedness]\label{thm:LWP} Let $s>d/2+1$, $\mu\leq 1$ and assume that \eqref{condition_initiala}, \eqref{uniform_from_above_below} and \eqref{perturbations} hold. There exists $\varepsilon_{0}=\varepsilon_{0}(K,M)>0$, $\eta_{0}=\eta_{0}(K,M)>0$ such
that if
\[ \text{ }\varepsilon\leq\varepsilon_{0}
\text{ and }\left\Vert (\delta p_{0}^{\varepsilon,\mu},\delta\theta_{0}^{\varepsilon,\mu
})\right\Vert _{H^{s-1}}\leq\eta_{0},
\]
then there exists $T=T(K,M)$ independent of $\mu$ and $\varepsilon$ such that
the Cauchy problem associated to the system \eqref{BN} with initial data
$\left(  \alpha_{+,0}^{\varepsilon,\mu},\rho_{\pm,0}^{\varepsilon,\mu}%
,\theta_{\pm,0}^{\varepsilon,\mu},u_{0}^{\varepsilon,\mu}\right)  $ defined by \eqref{condition_initiala}-\eqref{real_initial_data}admits a unique solution
\[
\left(  \alpha_{+}^{\varepsilon,\mu}-\bar{\alpha}_{+},\rho_{\pm}%
^{\varepsilon,\mu}-\bar{\rho}_{\pm},\theta_{\pm}^{\varepsilon,\mu}-\bar{\theta},u^{\varepsilon,\mu}\right)  \in C([0,T];(H^{s}%
(\mathbb{R}^{d}))^{5+d})\cap C^{1}((0,T);(H^{s-1}(\mathbb{R}^{d}))^{5+d}).
\]
Moreover, there exists $C>0$ independent of $\varepsilon,\mu$ such that
\begin{gather}
\left\Vert \left(  \alpha_{+}^{\varepsilon,\mu}-\bar{\alpha}_{+}%
,\rho_{\pm}^{\varepsilon,\mu}-\bar{\rho}_{\pm},\theta_{\pm}%
^{\varepsilon,\mu}-\bar{\theta},u^{\varepsilon,\mu}\right)  \right\Vert
_{L^{\infty}(\left(  0,T\right)  ;H^{s})}^{2}+\frac{\Vert\delta p^{\varepsilon
,\mu}\Vert_{L^{2}_T(\left(  0,T\right)  ;H^{s})}^{2}}{\mu\varepsilon
}+\frac{\Vert\delta\theta^{\varepsilon,\mu}\Vert_{L_{T}^{2}(\left(
0,T\right)  ;H^{s})}^{2}}{\mu}\nonumber\\
\text{
\ \ \ \ \ \ \ \ \ \ \ \ \ \ \ \ \ \ \ \ \ \ \ \ \ \ \ \ \ \ \ \ \ \ \ \ \ \ \ \ \ \ \ \ \ \ \ \ \ \ }%
\leq C\Vert(\alpha_{+,0}^{\varepsilon,\mu}-\bar{\alpha}_{+}%
,p_{0}^{\varepsilon,\mu}-\bar{p},\theta_{0}^{\varepsilon,\mu}-\bar{\theta},\delta p_{0}^{\varepsilon,\mu},\delta\theta_{0}^{\varepsilon,\mu})\Vert_{H^{s}}^{2} \label{bound1}
\end{gather}
and
\begin{gather}\label{bound2}
\frac{\Vert\delta p^{\varepsilon
,\mu}\Vert_{L^{1}_T(\left(  0,T\right)  ;H^{s-1})}}{\mu\varepsilon
}+\frac{\Vert\delta\theta^{\varepsilon,\mu}\Vert_{L_{T}^{1}(\left(
0,T\right)  ;H^{s-1})}}{\mu}\leq C(M).
\end{gather}
\end{theorem}
The previous result is the first step in order to study the limiting behavior of sequences of solutions of
\eqref{BN} when $\mu\rightarrow0$. We emphasize that the estimates are uniform with respect to both $\varepsilon$ and $\mu$. If we identify variables for which the time derivative is controlled then the Arzelà-Ascoli theorem can be used to justify strong convergence. In particular, one may either fix $\mu$ and analyze the asymptotic behavior as $\varepsilon \to 0$, or fix $\varepsilon$ and consider the limit $\mu \to 0$, or else prescribe a specific scaling relation between $\varepsilon$ and $\mu$. The pressure relaxation limit has already been a subject of our previous contributions \cite{BurteaCrin-BaratTan2023,BurteaCrin-BaratGonin-Joubert2026}. Here we focus on $\mu\rightarrow0$.

We consider a family of initial data defined by
\eqref{condition_initiala}-\eqref{real_initial_data} and assume that 
\begin{equation}
\Vert(\alpha_{+,0}^{\varepsilon,\mu}-\bar{\alpha}_{+}%
,p_{0}^{\varepsilon,\mu}-\bar{p},\theta_{0}^{\varepsilon,\mu}%
-\bar{\theta},\delta p_{0}^{\varepsilon,\mu},\delta\theta_{0}%
^{\varepsilon,\mu})\Vert_{H^{s}}\leq C_0.
\end{equation}
Let
\begin{equation}
\rho_{0}^{\varepsilon,\mu}=\alpha_{+,0}^{\varepsilon,\mu}\rho_{+,0}%
^{\varepsilon,\mu}+\alpha_{-,0}^{\varepsilon,\mu}\rho_{-,0}^{\varepsilon,\mu
},\text{ }y_{\pm,0}^{\varepsilon,\mu}=\frac{\alpha_{+,0}^{\varepsilon,\mu}%
\rho_{+,0}^{\varepsilon,\mu}}{\rho_{0}^{\varepsilon,\mu}}\label{dens_lim}%
\end{equation}
and
\begin{equation}
w_{0}^{\varepsilon,\mu}:=\frac{y_{+,0}^{\varepsilon,\mu}R_{+}+y_{-,0}%
^{\varepsilon,\mu}R_{-}}{y_{+,0}^{\varepsilon,\mu}\cv_{+}+y_{-,0}%
^{\varepsilon,\mu}\cv_{-}}\left(  y_{+,0}^{\varepsilon,\mu}\cv_{+}%
\theta_{+,0}^{\varepsilon,\mu}+y_{-,0}^{\varepsilon,\mu}\cv_{-}\theta
_{-}^{\varepsilon,\mu}\right)  .\label{dens_lim_2}%
\end{equation}
Denote also
\begin{equation}
\bar{w}=(\bar{y}_{+}R_{+}+\bar{y}_{-}R_{-})\bar{\theta},\label{w_la_limita}%
\end{equation}
the far-field behavior of $w_{0}^{\varepsilon,\mu}$. We will assume that
\begin{equation}
\left(  \rho_{0}^{\varepsilon,\mu}-\bar{\rho},y_{+,0}^{\varepsilon,\mu}-\bar{y}_{+},w_{0}^{\varepsilon,\mu}-\bar{w},u_0^{\varepsilon,\mu}\right)  \underset{\mu,\varepsilon\to0}{\to} \left(
\rho_{0}-\bar{\rho},y_{+,0}-\bar{y}_+,w_{0}-\bar{w},u_0\right)  \text{ in }%
(H^{s}(  \mathbb{R}^{d})  )^{3+d}.\label{convergence}%
\end{equation}

The system \eqref{Kapp_sym} is a quasilinear system of PDEs of order $1$ that is
symmetrizable with a diagonal matrix as we explained in Sections \ref{sec:sym}.
The local-in-time existence of a unique solution with initial data satisfying
\begin{equation}
\left(  \rho_{0},y_{+,0},w_{0},u_0\right)  \in\left(  \bar{\rho},\bar{y}_+,\bar
{w},0_{\mathbb{R}^d}\right)  +(H^{s}(  \mathbb{R}^{d})  )^{3+d}\label{ini_data_limit}%
\end{equation}
is clear. 

In our second main result, we formalize the fact that the unique solution to \eqref{Kapp_sym} can be obtained as the limit of the solution to \eqref{BN} as $\mu \to 0$.

\begin{theorem}
[Temperature and pressure relaxation limit]\label{thm:relax} Assume that the assumptions and
notations of Theorem \ref{thm:LWP} are in place. Assume furthermore, that the
initial data satisfy \ref{dens_lim}, \ref{dens_lim_2} and \ref{convergence}
and let
\[
\left(  \rho,y_{+},w,u\right)  \in\left(  \bar{\rho},\bar{y}_{+},\bar
{w},0_{\mathbb{R}^{d}}\right)  +(H^{s}(\mathbb{R}^{d}))^{3+d}%
\]
be the unique solution of the system \eqref{Kapp_sym} with initial data $\left(
\text{\ref{ini_data_limit}}\right)$. There exists $C>0$ depending only on $K$ and $M$ such that
\begin{align*}
&  \Vert(\rho^{\varepsilon,\mu}-\rho,y_{+}^{\varepsilon,\mu}-y_{+}%
,w^{\varepsilon,\mu}-w,u^{\varepsilon,\mu}-u)\Vert_{L^{\infty}(\left(
0,T\right)  ;H^{s-1})}\\
&  \hspace{5cm}\leq C\sqrt{\mu}+\Vert(\rho_{0}^{\varepsilon,\mu}-\rho_{0},
y_{+,0}^{\varepsilon,\mu}-y_{0},
w_{0}^{\varepsilon,\mu}-w_{0},u_{0}^{\varepsilon,\mu}-u_{0})\Vert_{H^{s-1}}
\end{align*}
where $T$ is the common time of existence of \eqref{Kapp_sym} and \eqref{BN}.
Moreover, we have
\begin{align*}
&  \Vert(\rho^{\varepsilon,\mu}-\rho,y_{+}^{\varepsilon,\mu}-y_{+}%
,w^{\varepsilon,\mu}-w,u^{\varepsilon,\mu}-u)\Vert_{L^{\infty}(\left(
0,T\right)  ;H^{s-2})}\\
&  \hspace{5cm}\leq C\mu+\Vert(\rho_{0}^{\varepsilon,\mu}-\rho_{0},
y_{+,0}^{\varepsilon,\mu}-y_{0},
w_{0}^{\varepsilon,\mu}-w_{0},u_{0}^{\varepsilon,\mu}-u_{0})\Vert_{H^{s-2}}.
\end{align*}

\end{theorem}

\begin{remark}
As highlighted in Section \ref{sec:sym}, the core of our proof lies in rewriting the system \eqref{BN} in a normal form that effectively decouples it into two subsystems: one governing the fast variables, and another governing the slow variables. The latter retains a structure that is asymptotically close to that of the limiting system as $\varepsilon,\mu \to 0$. The main difference compared with Yong's normal form structure \cite{Yong1999} is that the coupling through singular terms appears only in the subsystem associated with the fast variables, while the subsystem governing the slow variables involves only $O(1)$ coupling terms with respect to $\varepsilon$ and $\mu$. More precisely, recalling that the general form of \eqref{BN} written for $U=(U_1,U_2,w,\rho,y,u)$ is
\begin{equation}\label{eq:canon}
\partial_t U + \sum_{j=1}^d H^j(U)\partial_{x_j} U + D_{\varepsilon,\mu}U = R_{\varepsilon,\mu},
\end{equation}
we observe that if we denote the solution of \eqref{Kapp_sym} by $\tilde{U}^{0} =\left(w,\rho,y,u\right)\in\mathbb{R}^{3+d}$, then
\[
\partial_{t}\tilde{U}^{0}+\sum_{j=1}^{d}\tilde{H}^{j}(  (0,0,\tilde
{U}^{0}))  \partial_{x_{j}}\tilde{U}^{0}=0,
\]
where $\tilde{H}^j$ is the lower-right block of $H^j$ obtained after eliminating the first two rows and columns.
This explains why our approach allows us to obtain the result without resorting to higher-order expansions, as done in \cite{Yong1999,GiovangigliYong2018}.
\end{remark}

\section{Proof of Theorem \ref{thm:LWP}}

The \eqref{BN} system can be formally written as
\[
\partial_{t}V+\sum_{j=1}^{d}A^j\left(  V\right)  \partial_{x_{j}}%
V=\frac{Q_{1}\left(  V\right)  }{\varepsilon\mu}+\frac{Q_{2}\left(  V\right)
}{\mu},%
\]
where $A^j,Q_{1},Q_{2}:\mathbb{R}^{5+d}\rightarrow\mathcal{M}_{5+d,5+d}%
\left(  \mathbb{R}\right)  $. The discussion in Sections $2$ and $3$ shows
that we can define a function
\[
U:(0,1)\times(0,+\infty)^{4}\times\mathbb{R}^{d}\rightarrow\mathbb{R}%
^{5}\times\mathbb{R}^{d},
\]
such that formally $U=U\left(  V\right)  $ verifies system $\left(
\text{\ref{eq:canon}}\right)  $ where for all $j\in\overline{1,d}$,
\[
H_{j}:(0,1)\times(0,+\infty)^{4}\rightarrow\mathcal{M}_{5+d,5+d}\left(
\mathbb{R}\right)  .
\]
As we explained, $\left(  \text{\ref{eq:canon}}\right)  $ admits a positive-definite diagonal
symmetrizer provided that the prefactors of $\operatorname{div}u$ appearing in
the equations of $\left(  U_{1},U_{2},w\right)  $ of $\left(
\text{\ref{eq:canon}}\right)  $ are positive. These prefactors are given
respectively by
\begin{align*}
s_{11}  & :=f_{11}+O\left(  U_{1},U_{2}\right)  =\frac{\left(  \gamma
_{+}-\gamma_{-}\right)  ^{2}}{\left\langle R\right\rangle \left\langle
\frac{R}{\gamma}\right\rangle }\frac{y_{+}R_{+}}{\gamma_{+}}\frac{y_{-}R_{-}%
}{\gamma_{-}}\frac{\left\langle R\right\rangle }{\left\langle \cv%
\right\rangle }\left\langle e\right\rangle +O\left(  U_{1},U_{2}\right)  ,\\
s_{22}  & :=f_{22}+O\left(  U_{1},U_{2}\right)  =\frac{\left(  \gamma
_{+}-\gamma_{-}\right)  ^{2}}{\left\langle \cv\right\rangle \left\langle
\gamma R\right\rangle }\cv_{+}\cv_{-}y_{+}y_{-}\frac{\left\langle
R\right\rangle }{\left\langle \cv\right\rangle }\left\langle e\right\rangle
+O\left(  U_{1},U_{2}\right)  ,\\
s_{33}  & :=f_{33}+O\left(  U_{1},U_{2}\right)  =\left(\frac{\left\langle
R\right\rangle }{\left\langle \cv\right\rangle }\right)^2\left\langle e\right\rangle +O\left(  U_{1}%
,U_{2}\right)  ,
\end{align*}
where we recall that $O(U_{1},U_{2})$ denotes generic terms of the same order as $\lvert U_{1}\rvert + \lvert U_{2}\rvert$, and that for $i \in \{1,2,3\}$, the following functions are strictly positive:
\[
f_{ii}:\text{ }(0,+\infty)^{2}\times(0,1)\rightarrow(0,+\infty).
\]

For the \eqref{BN} system, this means that we can fix
without any constraints the initial data for the "mean-value"-profile $\left(
\alpha_{+,0},p_{0},\theta_{0}\right)  \in K\subset\subset(0,1)\times
(0,+\infty)^{2}$. There exists a $\eta_0'$ depending on $K$ such that if the initial data for the pressure and temperature differences $\delta p_{0},\delta\theta_{0}$
are such that 
\[||(\delta p_{0},\delta\theta_{0})||_{H^{s-1}}\leq\eta_0',\] then $\theta
_{\pm,0},\rho_{\pm,0}$ defined by \eqref{real_initial_data} and
for $i\in\overline{1,3},$ $d_{ii}|_{t=0}$ defined above are strictly
positive. We can thus find a compact set $K' \subset \mathcal{O} := (0,1) \times (0,+\infty)^4$ such that 
$V'_0 := (\alpha_{+,0}, \rho_{\pm,0}, \theta_{\pm,0}) \in K'$. Denoting $K''=\{V'\in\mathcal{O} \text{ }:\text{ }\operatorname{dist}(V',K')\leq\frac{1}{2}\operatorname{dist}(K',\mathcal{O}^c)\}$, we see that the last parameter to take into account for how small $\delta
p_{0},\delta\theta_{0}$ should be chosen is given by $\eta_{K''}$ defined in Proposition \ref{change_of_var11}.

By the classical theory of quasilinear symmetrizable hyperbolic systems (see for example \cite{BenzoniGavage2007,bahouriFourierAnalysisNonlinear2011}), there exists some $T_{\varepsilon,\mu}>0$ such that \eqref{BN} with initial condition given by \eqref{condition_initiala}-\eqref{real_initial_data} has a unique solution 
$V^{\varepsilon,\mu}\in C^0(0,T,\bar{V}+H^s(\mathbb{R}^d))\cap C^1(0,T,\bar{V}+H^{s-1}(\mathbb{R}^d))$ satisfying
\begin{equation}
    \text{for all }(t,x)\in [0,T_{\varepsilon,\mu}[\times\R^n,\quad V^{\varepsilon,\mu}(t,x)\in K''_{\eta(K'')},
\end{equation}
where $K''_{\eta(K'')}$ is defined in Proposition \ref{change_of_var11}, i.e., the time of existence is chosen such that the change of variable from $V$ to $U$ defines a diffeomorphism. We can thus focus on analyzing system \eqref{eq:canon}. 

Denote by $G=G(U)\in \mathcal{M}_{5+d,5+d}(\mathbb{R})$ the diagonal symmetrizer for system \eqref{eq:canon}.
Denote for all $t\in [0,T_{\varepsilon,\mu}[$,
\begin{equation}
    X^2(t) = \sup_{[0,t]}\,\langle G(U) \Lambda^s(U-\bar{U}),\Lambda^s(U-\bar{U})\rangle + \frac{\widetilde{\underline{a_1}}\Vert U_1\Vert_{L_t^2(H^s)}^2}{\mu\varepsilon} + \frac{\underline{\widetilde{a_2}}\Vert U_2\Vert_{L_t^2(H^s)}^2}{\mu},
\end{equation}
where we denote, for $i=1,2$,
\begin{equation}
    \text{for all }t\in [0,T],\quad\widetilde{\underline{a_i}}(t) = \inf_{s\in [0,t],x\in\R^d} G_{ii}(U(s,x))a_i(U(s,x)) >0.
\end{equation}
Consider the maximal time $T\leq T_{\varepsilon,\mu}$ with the property that 
\begin{equation}
\left\{
\begin{array}
[c]{l}%
X^{2}\left(  t\right)  \leq3X^{2}\left(  0\right)  \text{ for all }t\in
\lbrack0,T],\\
\text{for }i\in\overline{3,5}\text{, }\inf\limits_{[0,T]\times\mathbb{R}^{d}%
}U_{i}\left(  t,x\right)  \geq\frac{1}{3}\inf\limits_{\mathbb{R}^{d}}%
U_{i}\left(  0,x\right)  ,\\
\text{for }i\in\overline{1,2}\text{, }\sup\limits_{\left[  0,T\right]
}\left\Vert U_{i}\left(  t\right)  \right\Vert _{H^{s}}\leq2\eta.
\end{array}
\right.  \label{bootstrap}%
\end{equation}
We define
\[
M:=\max\{X^{2}\left(  0\right)  ,\frac{1}{\inf\limits_{\mathbb{R}^{d}}%
U_{i}\left(  0,x\right)  },i\in\overline{3,5}\},\\
\tilde{M}:=||U||_{L^\infty}
\]
By a slight abuse of notation, we still denote this maximal time by $T_{\varepsilon,\mu}$. Using energy estimates and a bootstrap argument we will show that $T_{\varepsilon,\mu}$ admits a lower bound that is independent of the singular parameters.

In all the following, we denote by $C(M)>0$ some generic constant depending only on $d,s$ and $M$, which can change at each line. Note that if $\varphi\in C^\infty_c(\R^n;\R)$ and $\Psi\in H^s(\R^d;\R)$, we have (as $s>d/2$)
\begin{equation}
\Vert\varphi(U)\Psi\Vert_{H^s} \leq \Vert \varphi(U)-\varphi(\bar{U})\Vert_{H^s}\Vert\Psi\Vert_{H^s} + \varphi(\bar{U})\Vert\Psi\Vert_{H^s} \leq C(M)\Vert\Psi\Vert_{H^s}.
\end{equation}
Moreover, if $j\in \{1,\cdots d\}$,
\begin{equation}
    \Vert\partial_{x_j} U\Vert_{H^{s-1}} = \Vert\partial_{x_j}(U-\bar{U})\Vert_{H^{s-1}}\leq \Vert U - \bar{U}\Vert_{H^s}\leq C(M).
\end{equation}
For $i=1,2$, we denote
\begin{equation}
    \text{for all }t\in [0,T],\quad\underline{a_i}(t) = \inf_{s\in [0,t],x\in\R^d} a_i(U(s,x)) >0.
\end{equation}
Note that, from \eqref{bootstrap}, we get $1/\underline{a_i}\leq C(M)$.
In order to ease the reading of the paper, we assume that the regularity index $s$ is an integer. The proof extends without difficulty to the case where $s$ is real.

Let us define 
\begin{equation}
    Q=\{\mathbf{q}=(q_{0,1},q_{0,2},\cdots,q_{s-1,1},q_{s-1,2},q_{\infty,1},q_{\infty,2})\in \R^{2s+2}, q_{i,j}\geq 1\}.
\end{equation}
For all $\mathbf{q}\in Q$, we define
\begin{equation}
\begin{split}
    Y_\mathbf{q}(t) &=q_{\infty,1}\left(\Vert  U_1\Vert_{L^\infty_t(L^\infty)}+\frac{\int_0^t \underline{a_1}\Vert U_1\Vert_{L^\infty}}{2\mu\varepsilon}\right) +q_{\infty,2}\left(\Vert  U_2\Vert_{L^\infty_t(L^\infty)}+\frac{\int_0^t\underline{a_2}\Vert  U_2\Vert_{L^\infty}}{2\mu}\right) \\&+ \sum_{k=0}^{s-1} q_{k,1}\left(\Vert  U_1\Vert_{L^\infty_t(H^k)}+ \frac{\int_0^t\underline{a_1}\Vert U_1\Vert_{H^k_x}}{2\mu\varepsilon}\right) +\sum_{k=0}^{s-1}q_{k,2}\left(\Vert U_2\Vert_{L^\infty_t(H^k)}+\frac{\int_0^t\underline{a_2}\Vert U_2\Vert_{H^k_x}}{2\mu}\right).
    \end{split}
\end{equation}
The following proposition holds:
\begin{prop}\label{prop:etapetit}
There exists some $\mathbf{q}=\mathbf{q}(M)\in Q$ such that for all $t\in[0,T^{\varepsilon,\mu}[$
\begin{equation}\label{eq:ineqYq}
    Y_{\mathbf{q}}(t)\leq Y_{\mathbf{q}}(0) + C(M)t +  C(M)\varepsilon Y_{\mathbf{q}}(t) + C(M)Y_{\mathbf{q}}(t)^2.
\end{equation}
\end{prop}
\begin{proof}
Let $\mathbf{q}\in Q$, $\vert \mathbf{q}\vert\leq C(M)$ to be chosen.
From \eqref{eq:canon}, we get
\begin{equation}\label{eq:prems}
    \partial_t U_1 + \frac{a_1 U_1}{\mu\varepsilon} = \frac{U_1}{\mu\varepsilon}r_{11}(U) + \frac{U_2}{\mu}r_{21}(U) + \frac{U_1}{\mu}r_{31}(U) - \sum_{j=1}^{d}(H^j\partial_{x_j} U)_1,
\end{equation}
hence, multiplying by $q_{\infty,1}$ gives
\begin{equation}
    \partial_t (q_{\infty,1} U_1) + \frac{a_1 q_{\infty,1} U_1}{\mu\varepsilon} = \frac{U_1}{\mu\varepsilon}q_{\infty,1} r_{11}(U) + \frac{U_2}{\mu} q_{\infty,1} r_{21}(U) + \frac{U_1}{\mu} q_{\infty,1} r_{31}(U) - q_{\infty,1}\sum_{j=1}^{d}(H^j\partial_{x_j} U)_1.
\end{equation}
Using the maximum principle we deduce that
\begin{align*}
    q_{\infty,1} \Vert U_1\Vert_{L^\infty} + \frac{q_{\infty,1}\underline{a_1}\int_0^t \Vert U_1\Vert_{L^\infty}}{\mu\varepsilon} &\leq q_{\infty,1}\Vert U_1^0\Vert_{L^\infty} + C_{\infty,1}(M)q_{\infty,1} \frac{\int_0^t \Vert U_2\Vert_{L^\infty}}{\mu}\\&\quad + C(M)Y_{\mathbf{q}}^2 + \varepsilon C(M)Y_{\mathbf{q}} + C(M)t.
    \end{align*}
By a similar reasoning, we get
\begin{equation}
    q_{\infty,2} \Vert U_2\Vert_{L^\infty} + \frac{q_{\infty,2}\underline{a_2}\int_0^t \Vert U_2\Vert_{L^\infty}}{\mu\varepsilon} \leq q_{\infty,2}\Vert U_2^0\Vert_{L^\infty} + C(M)Y_{\mathbf{q}}^2 + \varepsilon C(M)Y_{\mathbf{q}} + C(M)t.
\end{equation}
Let $k\in \{0,\cdots,s-1\}$. Applying $\Lambda^k$ to \eqref{eq:prems}, we get
\begin{equation}\label{eq:toutoutderivbis}
\begin{split}
&\partial_t \Lambda^k U_1 + \frac{a_1 \Lambda^k U_1}{\mu\varepsilon} = - \sum_{j=1}^d\Lambda^k(H^j\partial_{x_j} U)_1+ \frac{\Lambda^k U_1}{\mu\varepsilon}r_{11} + \frac{\Lambda^k U_2}{\mu}r_{21} + \frac{\Lambda^k U_1}{\mu}r_{31}
\\& + [a_1,\Lambda^k]\frac{U_1}{\mu\varepsilon} - [r_{11},\Lambda^k]\frac{U_1}{\mu\varepsilon} -  [r_{21},\Lambda^k]\frac{U_2}{\mu} -[r_{31},\Lambda^k]\frac{U_1}{\mu}.
\end{split}
\end{equation}
Multiplying by $\overline{\Lambda^k U_1}$ then integrating on $\R^n$, we obtain
\begin{equation}
\begin{split}
    \frac{1}{2}\frac{d}{dt} (\Vert U_1\Vert_{H^k}^2) + \frac{a_1\Vert U_1\Vert_{H^k}^2}{\mu\varepsilon} &\leq C(M)\Vert U_1\Vert_{H^k} + C(M)\frac{\Vert U_1\Vert_{H^k}^2}{\mu\varepsilon}Y_{\mathbf{q}} + C_{k,1}(M)\frac{\Vert U_2\Vert_{H^k}}{\mu}\Vert U_1\Vert_{H^k} + \varepsilon C(M) \frac{\Vert U_1\Vert_{H^k}^2}{\mu\varepsilon} 
    \\&+ C_{k}^1(M)\left(\frac{\Vert U_1\Vert_{H^{k-1}}}{\mu\varepsilon} + \frac{\Vert U_2\Vert_{H^{k-1}}}{\mu} + \frac{\Vert U_1\Vert_{L^\infty}}{\mu\varepsilon} + \frac{\Vert U_2\Vert_{L^\infty}}{\mu}\right)\Vert U_1\Vert_{H^k}. 
\end{split}
\end{equation}
Using Lemma \ref{SimpliCarre}, we deduce
\begin{equation}
\begin{split}
    q_{k,1}\Vert U_1\Vert_{H^k} &+ \frac{q_{k,1}\underline{a_1}\int_0^t \Vert U_1\Vert_{H^k}}{\mu\varepsilon} \leq q_{k,1}\Vert U_1^0\Vert_{H^k} + C(M)t + C(M)Y_{\mathbf{q}}^2 + \varepsilon C(M)Y_{\mathbf{q}} + C_{k,1}(M)q_{k,1}\frac{\int_0^t\Vert U_2\Vert_{H^k}}{\mu} \\&+q_{k,1}C_{k,1}(M)\left(\frac{\int_0^t\Vert U_1\Vert_{H^{k-1}}}{\mu\varepsilon} + \frac{\int_0^t\Vert U_2\Vert_{H^{k-1}}}{\mu} + \frac{\int_0^t\Vert U_1\Vert_{L^\infty}}{\mu\varepsilon} +\frac{\int_0^t\Vert U_2\Vert_{L^\infty}}{\mu}\right).
\end{split}
\label{U_1_Hk}
\end{equation}
Similarly,
\begin{equation}
    \begin{split}
    q_{k,2}\Vert U_2\Vert_{H^k} &+\frac{q_{k,2}\underline{a_2}\int_0^t\Vert U_2\Vert_{H^k}}{\mu} \leq q_{k,2}\Vert U_2^0\Vert_{H^k} + C(M) t + C(M)Y_{\mathbf{q}}^2+\varepsilon C(M)Y_{\mathbf{q}} \\&+q_{k,2}C_{k,2}(M)\left(\frac{\int_0^t\Vert U_1\Vert_{H^{k-1}}}{\mu\varepsilon} + \frac{\int_0^t\Vert U_2\Vert_{H^{k-1}}}{\mu} + \frac{\int_0^t\Vert U_1\Vert_{L^\infty}}{\mu\varepsilon} + \frac{\int_0^t\Vert U_2\Vert_{L^\infty}}{\mu}\right).
\end{split}
\label{U_2_Hk}
\end{equation}
The above two estimates shed some light on the constants $(q_{k,i})_{k\in\overline{0,s-1}, i\in\overline{1,2}}$: they are chosen such that the damping term from the l.h.s. of the inequality at the level $k-1$ absorbs the corresponding r.h.s. terms of the inequality at the level $k$.
For instance, we may choose
\[
\begin{cases}
    q_{s-1,1} = 1,&\\
    q_{s-1,2} = \dfrac{2}{\underline{a_2}}q_{s-1,1}C_{s-1,1}+1,&\\ 
    q_{k-1,1} =  \dfrac{2}{\underline{a_1}}(q_{k,1}C_{k,1} + q_{k,2}C_{k,2})+1 & \text{for }k=1,\cdots,s-1,\\
    q_{k-1,2} =  \dfrac{4}{\underline{a_2}}(q_{k,1}C_{k,1} + q_{k,2}C_{k,2} + q_{k-1,1}C_{k-1,1}/2)+1 & \text{for }k=1,\cdots,s-1,\\
    q_{\infty,1} =\dfrac{2}{\underline{a_1}}\sum_{i=0}^{s-1} (q_{i,1}C_{k,1}+q_{i,2}C_{k,2}) + 1,&\\
     q_{\infty,2} =\dfrac{4}{\underline{a_2}}\sum_{i=0}^{s-1} (q_{i,1}C_{k,1}+q_{i,2}C_{k,2}) + \dfrac{2}{\underline{a_2}}q_{\infty,1}C_{\infty,1}+ 1.&
\end{cases}
\]
This ends the proof of Proposition \ref{prop:etapetit}.
\end{proof}
In the following of this section, let us fix $\mathbf{q}\in Q$ such that \eqref{eq:ineqYq} holds. 
Applying $\Vert\cdot\Vert_{L^\infty_t(H^{s-1})}$ to \eqref{eq:canon}, it is straightforward to show the following proposition.
\begin{prop}\label{eq:dtU}For all $t\in[0,T^{\varepsilon,\mu}[$, we have
\begin{equation}
    \Vert \partial_t U\Vert_{L^1_t(H^{s-1})}\leq C(M)t + C(M)Y_{\mathbf{q}}(t).
\end{equation}
As a consequence, for $i\in\overline{3,5}$ and for all $(t,x)\in [0,T^{\varepsilon,\mu}[\times\R^d$,
\begin{equation}
    U_i(t,x)\geq  \frac{1}{M} - C(M)t - C(M)Y_{\mathbf{q}}(t).
\end{equation}
\end{prop}

\begin{prop}\label{prop:Xt}There exists $\eta_0$ depending only on $\min_{i=3,4,5}\inf_{\R^d} U_i(0,x)$ such that if $\eta\leq \eta_0$ then, for all $t\in [0,T^{\varepsilon,\mu}[$, we have
\begin{equation}
    X(t)^2 \leq X(0)^2 + \frac{1}{2}M +  C(M)(t + \varepsilon + Y_{\mathbf{q}}(t)).
\end{equation}
\end{prop}
\begin{proof} Recall that $G=G(U)\in \mathcal{M}_{5+d}(\mathbb{R})$ is the diagonal symmetrizer for system \eqref{eq:canon}.
For a right choice of $\eta_0$, the system \eqref{eq:canon0} can be symmetrized by multiplying it by $G$, with $G_{ii}\geq (1+C(M))^{-1}$ for $i\in\overline{1,5}$ (see Section \ref{subsec:sym}). Multiplying \eqref{eq:canon} by $G$, we get
\begin{equation}\label{eq:sym}
    G\partial_tU + \sum_{i=1}^d S^j \partial_{x_j}U + \widetilde{D}_{\varepsilon,\mu}U= \widetilde{R}_{\varepsilon,\mu},
\end{equation}
where $S^j = GH^j$ are symmetric, $\widetilde{D}_{\varepsilon,\mu} = \diag(G_{11}a_{1}/(\mu\varepsilon), G_{22}a_{2}/\mu,0,\cdots 0)$ is positive definite and $\widetilde{R}_{\varepsilon,\mu}=(G_{11}{R_{\varepsilon,\mu}}_{1}, ...,G_{55}{R_{\varepsilon,\mu}}_{5})$. Applying $\Lambda^s$ to \eqref{eq:sym}, we get
\begin{equation}
    G\partial_t \Lambda^s U + \sum_{i=1}^d S^j \partial_{x_j}\Lambda^s U + \widetilde{D}_{\varepsilon,\mu} \Lambda^s U = \Lambda^s \widetilde{R}_{\varepsilon,\mu} + [G,\Lambda^s]\partial_t U + \sum_{j=1}^d [S^j,\Lambda^s]\partial_{x_j}U + [\widetilde{D}_{\varepsilon,\mu},\Lambda^s]U.
\end{equation}
Remark that 
\begin{equation}
    \widetilde{D}_{\varepsilon,\mu}\Lambda^s U = \widetilde{D}_{\varepsilon,\mu}\Lambda^s(U-\bar{U}),\quad [\widetilde{D}_{\varepsilon,\mu},\Lambda^s]U = [\widetilde{D}_{\varepsilon,\mu},\Lambda^s](U-\bar{U})
\end{equation}
due to $\bar{U}_1 = \bar{U}_2=0$ and the particular structure of $D_{\varepsilon,\mu}$.
Applying $\langle \cdot,\Lambda^s (U-\bar{U})\rangle$ and using the symmetry of $G$ and $S^j$, we get
\begin{equation}
\begin{split}
    &\frac{1}{2}\frac{d}{dt}\langle G\Lambda^s (U-\bar{U}),\Lambda^s (U-\bar{U})\rangle + \int_{\R^d}\widetilde{a_1}\frac{\vert\Lambda^s U_1\vert^2}{\mu\varepsilon} + \int_{\R^d}\widetilde{a_2}\frac{\vert\Lambda^s U_2\vert^2}{\mu}
    \\&=\frac{1}{2}\langle\partial_t G\Lambda^s (U-\bar{U}),\Lambda^s( U-\bar{U})\rangle + \frac{1}{2}\sum_{j=1}^d\langle \partial_{x_j}S^j\Lambda^s(U-\bar{U}),\Lambda^s(U-\bar{U})\rangle
    \\&
     + \langle[G,\Lambda^s]\partial_t U,\Lambda^s (U-\bar{U})\rangle + \sum_{j=1}^d \langle [S^j,\Lambda^s]\partial_{x_j} U,\Lambda^s (U-\bar{U})\rangle + \frac{\langle [\widetilde{a_1},\Lambda^s]U_1,\Lambda^s U_1\rangle}{\mu\varepsilon} +\frac{\langle [\widetilde{a_2},\Lambda^s]U_2,\Lambda^s U_2\rangle}{\mu} 
     \\&+ \langle\Lambda^s \widetilde{R}_{\varepsilon,\mu},\Lambda^s(U-\bar{U})\rangle
     \\&=: I_1 + I_2 + I_3.
\end{split}
\end{equation}
Let us control each $\int_0^t I_i$, $i=1,2,3$.
By Hölder's inequality, we get
\begin{align}
    \left\vert \int_0^t I_1\right\vert &\leq \frac{1}{2}\Vert\partial_t G\Vert_{L^1_t(L^\infty)}\Vert U-\bar{U}\Vert_{L^\infty_t(H^s)}^2 + \frac{1}{2}t\sum_{j=1}^d \Vert \partial_x S^j\Vert_{L^\infty_t(L^\infty)}\Vert U-\bar{U}\Vert_{L^\infty_t(H^s)}^2 
    \\&\leq C(M)t + C(M)Y_{\mathbf{q}}(t).
\end{align}
Moreover, since $s>1+d/2$, we get from Proposition \ref{prop:LinfHs}, for $V\in H^{s-1}(\R^d)$,
\begin{equation}
    \Vert [f(U),\Lambda^s]V\Vert_{L^2} =\Vert [f(U)-f(\bar{U}),\Lambda^s]V\Vert_{L^2}\leq C\Vert f(U)-f(\bar{U})\Vert_{H^s}\Vert V\Vert_{H^{s-1}}\leq C(M)\Vert V\Vert_{H^{s-1}},
\end{equation}
hence
\begin{align}
    \left\vert \int_0^t I_2\right\vert&\leq C(M)\int_0^t\left(\Vert \partial_t U\Vert_{H^{s-1}} + 1 + \frac{\Vert U_1\Vert_{H^{s-1}}}{\mu\varepsilon} + \frac{\Vert U_2\Vert_{H^{s-1}}}{\mu}\right)\Vert U-\bar{U}\Vert_{H^s}
    \\&\leq C(M)t + C(M)Y_{\mathbf{q}}(t).
\end{align}
Finally,
\begin{align}
    I_3 &= \frac{\langle\widetilde{r_{11}}\Lambda^sU_1,\Lambda^sU_1\rangle}{\mu\varepsilon} + \frac{\langle \widetilde{r_{21}} \Lambda^sU_2,\Lambda^sU_1\rangle}{\mu} +  \frac{\langle\widetilde{r_{31}} \Lambda^sU_1,\Lambda^sU_1\rangle}{\mu}\\&+ \frac{\langle\widetilde{r_{12}}U_1\Lambda^sU_1,\Lambda^sU_2\rangle}{\mu\varepsilon} + \frac{\langle\widetilde{r_{22}} \Lambda^s U_2,\Lambda^sU_2\rangle}{\mu}  
    + \frac{\langle \widetilde{r_{32}}\Lambda^s U_1,\Lambda^sU_2\rangle}{\mu}
\\&-\frac{\langle[\widetilde{r_{11}},\Lambda^s]U_1,\Lambda^sU_1\rangle}{\mu\varepsilon} - \frac{\langle [\widetilde{r_{21}},\Lambda^s]U_2,\Lambda^sU_1\rangle}{\mu} -  \frac{\langle[\widetilde{r_{31}}, \Lambda^s]U_1,\Lambda^sU_1\rangle}{\mu}\\&- \frac{\langle[\widetilde{r_{12}}U_1,\Lambda^s]U_1,\Lambda^sU_2\rangle}{\mu\varepsilon} - \frac{\langle[\widetilde{r_{22}}, \Lambda^s] U_2,\Lambda^sU_2\rangle}{\mu}  
    - \frac{\langle [\widetilde{r_{32}},\Lambda^s] U_1,\Lambda^sU_2\rangle}{\mu}
    \\&=I_{3,1} - I_{3,2}.
\end{align}
By Hölder's inequality and Young's inequality, we first observe that
\begin{equation}
    \int_0^t \left\vert\frac{\langle \widetilde{r_{11}}\Lambda^s U_1,\Lambda^s U_1\rangle}{\mu\varepsilon}\right\vert \leq C(M)Y_{\mathbf{q}}(t)\int_0^t\frac{\Vert U_1\Vert_{H^s}^2}{\mu\varepsilon}\leq C(M)Y_{\mathbf{q}}(t).
\end{equation}
Next, we infer that
\begin{align}
    \int_0^t\left\vert \frac{\langle \widetilde{r_{21}}\Lambda^s U_2,\Lambda^s U_1\rangle}{\mu}\right\vert + \int_0^t \left\vert\frac{\langle\widetilde{r_{32}}\Lambda^sU_1,\Lambda^s U_2\rangle}{\mu}\right\vert &\leq C(M)\int_0^t \frac{\Vert U_1\Vert_{H^s}\Vert U_2\Vert_{H^s}}{\mu} \\&\leq \frac{\underline{\widetilde{a_1}}}{2}\int_0^t\frac{\Vert U_1\Vert_{H^s}^2}{\mu\varepsilon} + C(M)\varepsilon\int_0^t\frac{\Vert U_2\Vert_{H^s}^2}{\mu}
    \\&\leq \frac{1}{2}M + C(M)\varepsilon.
\end{align}
We have
\begin{align}
    \int_0^t \left\vert\frac{\langle\widetilde{r_{22}}\Lambda^s U_2,\Lambda^s U_2\rangle}{\mu}\right\vert \leq C(M)Y_{\mathbf{q}}(t)\int_0^t \frac{\Vert U_2\Vert_{H^s}^2}{\mu}\leq C(M)Y_{\mathbf{q}}(t).
\end{align}
Similarly,
\begin{align}
    \int_0^t\left\vert \frac{\langle\widetilde{r_{31}}\Lambda^sU_1,\Lambda^s U_1\rangle}{\mu}\right\vert \leq C(M)\varepsilon\int_0^t\frac{\Vert U_1\Vert_{H^s}^2}{\mu\varepsilon}\leq C(M)\varepsilon.
\end{align}
Finnaly, observe that
\begin{align}
    \int_0^1 \left\vert\frac{\langle \widetilde{r_{12}}U_1\Lambda^sU_1,\Lambda^sU_2\rangle}{\mu\varepsilon}\right\vert \leq C(M)\int_0^t \frac{\Vert U_1\Vert_{H^{s-1}}\Vert U_1\Vert_{H^s}\Vert U_2\Vert_{H^s}}{\mu\varepsilon} \leq C(M)Y_{\mathbf{q}}(t).
\end{align}
Note that the absence, in the equation for $U_2$, of any remainder term of the same order as $U_1U_2/(\mu\varepsilon)$ is crucial to close the estimates.
Moreover, we have
\begin{equation}
    \left\vert\int_0^t I_{3,2}\right\vert\leq C(M)Y_{\mathbf{q}}(t).
\end{equation}
Gathering the above estimates concludes the proof of Proposition \ref{prop:Xt}.
\end{proof}
Let $\eta>0$ such that
\begin{equation}
    Y_{\mathbf{q}}(0)\leq \eta,\quad Y_{\mathbf{q}}(T)\leq 3\eta.
\end{equation}
Employing Proposition \ref{prop:etapetit} and Proposition \ref{prop:Xt}, we deduce that there exists some $\overline{T}=\overline{T}(M)>0$, $\overline{\varepsilon}=\overline{\varepsilon}(M)>0$ and $\overline{\eta}=\overline{\eta}(M)>0$ such that, if $T\leq \overline{T}$, $\varepsilon\leq\overline{\varepsilon}$ and $\eta\leq\overline{\eta}$, then
\begin{equation}
    X(t)^2\leq 2M, \quad Y_{\mathbf{q}}(t)\leq 2\eta,\quad \text{for }i=3,4,5,\quad U_i\geq \frac{1}{2M}.
\end{equation} 
By the definition of $T^{\varepsilon,\mu}$ (see \eqref{bootstrap}), we have $T^{\varepsilon,\mu} \geq \overline{T}$. This completes the proof of Theorem \ref{thm:LWP}.
 \section{Proof of Theorem \ref{thm:relax}}

We observe that the system \eqref{Kapp_sym} shares some
structure with the system of equations $\left(  \text{\ref{eq:canon}}\right)  $
: if we denote $\tilde{U}^{0} =\left(w,\rho,y,u\right)\in\mathbb{R}^{3+d}$, then
\[
\partial_{t}\tilde{U}^{0}+\sum_{j=1}^{d}\tilde{H}^{j}\left(  (0,0,\tilde
{U}^{0})\right)  \partial_{x_{j}}\tilde{U}^{0}=0,
\]
where $\tilde{H}^j$ is the lower-right block of $H^j$ obtained after eliminating the first two rows and columns. We also have that
\[
\partial_{t}\tilde{U}^{\varepsilon,\mu}+\sum_{j=1}^{d}\tilde{H}^{j}\left(
U^{\varepsilon,\mu}\right)  \partial_{x_{j}}\tilde{U}^{\varepsilon,\mu}%
=\sum_{j=1}^{d}R_{1}\left(  U^{\varepsilon,\mu}\right)  \partial_{x_{j}}%
U_{1}^{\varepsilon,\mu}+\sum_{j=1}^{d}R_{2}\left(  U^{\varepsilon,\mu}\right)
\partial_{x_{j}}U_{2}^{\varepsilon,\mu},%
\]
with $R_{1}\left(  U^{\varepsilon,\mu}\right)  ,R_{2}\left(  U^{\varepsilon
,\mu}\right)  \in\mathbb{R}^{3+d}$. Defining
\[
W^{\varepsilon,\mu}:=\tilde{U}^{\varepsilon,\mu}-\tilde{U}^{0},%
\]
we observe that
\begin{align}
\partial_{t}W^{\varepsilon,\mu}+\sum_{j=1}^{d}\tilde{H}^{j}\left(
U^{\varepsilon,\mu}\right)  \partial_{x_{j}}W^{\varepsilon,\mu}  & =\sum
_{j=1}^{d}R_{1}\left(  U^{\varepsilon,\mu}\right)  \partial_{x_{j}}%
U_{1}^{\varepsilon,\mu}+\sum_{j=1}^{d}R_{2}\left(  U^{\varepsilon,\mu}\right)
\partial_{x_{j}}U_{2}^{\varepsilon,\mu}\nonumber\\
& +\sum_{j=1}^{d}\left(  \tilde{H}^{j}\left(  (0,0,\tilde{U}^{0})\right)
-\tilde{H}^{j}\left(  U^{\varepsilon,\mu}\right)  \right)  \partial_{x_{j}%
}\tilde{U}^{0}.\label{astea}%
\end{align}
Moreover, there exists a diagonal matrix $\tilde{G}\left(
U^{\varepsilon,\mu}\right)$ that symmetrizes $\left(
\text{\ref{astea}}\right)  $ which is just is the lower-right block of $G\left(
U^{\varepsilon,\mu}\right)$ obtained after eliminating the first two rows and columns. Thanks to these structural properties and the bounds \eqref{bound1}-\eqref{bound2} from the a priori estimates, the remainder of the proof follows from standard estimates, and we omit the details.

\subsection*{Acknowledgments}
This work was partially supported by a CNRS PEPS JCJC grant.
T. Crin-Barat is supported by the project ANR-24-CE40-3260 – Hyperbolic Equations, Approximations $\&$ Dynamics (HEAD) and the project ANR-25-CE40-5565 (Cookie).

 \medbreak
\noindent \textbf{Data availability statement} :
 Data sharing is not applicable to this article as no data sets were generated or analyzed during the current study.
 
\bigbreak
\noindent \textbf{Declarations}
 \medbreak
\noindent \textbf{Conflicts of interest} \,The authors have no competing interests to declare that are relevant to the content of this
article.



\section{Appendix} 

\begin{lemma}\label{SimpliCarre}
Let  $X : [0,T]\to \mathbb{R}_+$ be a continuous function such that $X^2$ is differentiable. We assume that there exists 
 a constant $B\geq 0$ and  a measurable function $A : [0,T]\to \mathbb{R}_+$ 
such that 
 $$\frac{1}{p}\frac{d}{dt}X^p+BX^p\leq AX^{p-1}\quad\hbox{a.e.  on }\ [0,T].$$ 
 Then, for all $t\in[0,T],$ we have
$$X(t)+B\int_0^tX\leq X_0+\int_0^tA.$$
\end{lemma}

\begin{prop}\label{prop:LinfHs} Suppose that $s>d/2+1$.
\begin{enumerate}
    \item \cite[Theorem 1.66]{bahouriFourierAnalysisNonlinear2011} There exists some $C=C(s,d)>0$ such that, for all $V\in H^{s-1}(\R^d)$,
    \begin{equation}
\Vert V\Vert_\infty\leq C\Vert V\Vert_{H^{s-1}}.
\end{equation}
\item \cite[Theorem 2.87]{bahouriFourierAnalysisNonlinear2011}
For all smooth function $f:\R\rightarrow\R$ and $g\in H^s(\R^d)$, there exists some non-decreasing function $\Phi:\R_+\rightarrow\R_+$ depending only on $s,d$ and $f'$ such that
\begin{equation}
    \Vert f\circ g\Vert_{H^s}\leq \Phi(\Vert g\Vert_{H^s}).
\end{equation}
\item \cite{katoCommutatorEstimatesEuler1988} Let us formally define the commutator $[\cdot,\cdot]$ by
\begin{equation}
[f,\Lambda^l]g = f\Lambda^l g - \Lambda^l fg.
\end{equation}
Then, for all $l>0$, there exists some $C=C(l,d)$ such that, for all $f\in H^{l}(\R^d)\cap W^{1,\infty}(\R^d)$, and $g\in H^{l-1}(\R^d)\cap L^\infty(\R^d)$, 
\begin{equation}\label{eq:magic0}
\Vert [f,\Lambda^l]g\Vert_2\leq C\Vert\nabla f\Vert_{\infty}\Vert g\Vert_{H^{l-1}} + C\Vert f\Vert_{H^l}\Vert g\Vert_{\infty}.
\end{equation}
In particular, if $f\in H^s(\R^d)$ and $l\leq s$,
\begin{equation}
    \Vert [f,\Lambda^l]g\Vert_2\leq C\Vert f\Vert_{H^s}(\Vert g\Vert_{H^{l-1}} + \Vert g\Vert_\infty),
\end{equation}
and if moreover $g\in H^{s-1}(\R^d)$,
\begin{equation}
    \Vert [f,\Lambda^l]g\Vert_2\leq C\Vert f\Vert_{H^s}\Vert g\Vert_{H^{s-1}}
    \label{eq:magic}.
\end{equation}
\end{enumerate}
\end{prop}

\printbibliography

\end{document}